\documentclass[a4paper,11pt]{article}

\usepackage[latin1]{inputenc}
\usepackage[T1]{fontenc}

\usepackage{lmodern}

\usepackage{natbib}
\renewcommand{\theenumi}{\roman{enumi}}
\newcommand{\comment}[1]{}

\usepackage{enumitem}\setlist[enumerate]{label={\rm(\theenumi)}}
\usepackage{graphicx,tikz}

\usepackage{amsmath,amssymb,amsthm}
\usepackage{mathtools}
\usepackage{mathrsfs} 

\numberwithin{equation}{section}

\theoremstyle{plain}
\newtheorem{theorem}{Theorem}[section]
\newtheorem{proposition}[theorem]{Proposition}

\newtheorem{lemma}[theorem]{Lemma}
\newtheorem{assumption}[theorem]{Assumption}

\theoremstyle{definition}
\newtheorem{definition}[theorem]{Definition}
\newtheorem{remark}[theorem]{Remark}

\newtheoremstyle{key}{3pt}{3pt}{}{}{\itshape}{:}{.5em}{}
\theoremstyle{key}
\newtheorem*{JEL}{JEL subject classification}
\newtheorem*{MSC}{MSC2010 subject classification}
\newtheorem*{keywords}{Keywords}

\DeclareMathOperator*{\esssup}{ess\,sup}

\renewcommand{\mid}{\:\vert\:}
\providecommand{\abs}[1]{\left\lvert#1\right\rvert}
\providecommand{\norm}[1]{\lVert#1\rVert}
\providecommand{\ceil}[1]{\left\lceil#1\right\rceil}

\newcommand{\nbd}[1]{\(#1\)\nobreakdash-\hspace{0pt}}
\newcommand{\indi}[1]{\mathbf{1}_{#1}}

\newcommand{\N}{\mathbb{N}}

\newcommand{\R}{\mathbb{R}}

\newcommand{\filt}[1]{\mathbf{#1}} 

\newcommand{\B}{{\mathscr B}}

\newcommand{\F}{{\mathscr F}}

\renewcommand{\S}{{\mathscr S}}
\newcommand{\T}{{\mathscr T}}

\newcommand{\cM}{{\cal M}}
\newcommand{\cP}{{\cal P}}

\begin{document}
\title{\textbf{Subgame-Perfect Equilibria in Stochastic Timing Games}\footnote{We would like to thank Jacco Thijssen for many valuable discussions. Financial support by the German Research Foundation (DFG) via grant Ri 1128-4-1, \textsl{Singular Control Games: Strategic Issues in Real Options and Dynamic Oligopoly under Knightian Uncertainty}, is gratefully acknowledged.}}
\author{
Frank Riedel \thanks{Center for Mathematical Economics, Bielefeld University, Germany; email: \texttt{friedel@uni-bielefeld.de}}
\and Jan-Henrik Steg \thanks{Center for Mathematical Economics, Bielefeld University, Germany; email: \texttt{jsteg@uni-bielefeld.de}}
}
\date{September 16, 2014}
\maketitle

{\textbf{Abstract.}}
We introduce a notion of subgames for stochastic timing games and the related notion of subgame-perfect equilibrium in possibly mixed strategies. While a good notion of subgame-perfect equilibrium for continuous-time games is not available in general, we argue that our model is the appropriate version for timing games. We show that the notion coincides with the usual one for discrete-time games.  Many timing games in continuous time have only equilibria in mixed strategies~-- in particular preemption games, which often occur in the strategic real option literature. We provide a sound foundation for some workhorse equilibria of that literature, which has been lacking as we show. We obtain a general constructive existence result for subgame-perfect equilibria in preemption games and illustrate our findings by several explicit applications.
\begin{JEL}
C61, C73, D21, L12
\end{JEL}
\begin{MSC}
60G40, 91A25, 91A40, 91A55
\end{MSC}

\begin{keywords}
timing games, stochastic games, mixed strategies, subgame-perfect equilibrium in continuous time, optimal stopping
\end{keywords}

\section{Introduction}\label{sec:intro}

The purpose of this work is to provide a framework for subgame-perfect equilibria in stochastic timing games. Therefore we first introduce a reasonable notion of \emph{subgames} for stochastic timing games. In continuous time there is no general clear concept of a subgame and it is much debated whether any such thing can be meaningful at all, compared to extensive form games\footnote{
See, for instance, \cite{Alos-FerrerRitzberger08}, who define an infinite generalization of a classical tree and show that it is in general necessary that every node has a clearly defined successor to guarantee unique outcomes of the game.
}. 
We argue that indeed there is a quite natural~-- although probably not obvious~-- concept for timing games, since one can characterize decision nodes by \emph{stopping times}. It is intuitively quite clear what the relevant information of a ``history'' in a timing game is: at what time which player has stopped, yet, and what has been revealed about the uncertain state of the world. Hence, the information about actions is very elementary and only exogenous information may cause complexity. Indeed, in continuous time the events that can be associated with some ``past'' or history are in general so rich that it is not enough to consider fixed deterministic dates and their information sets. This is why the notion of stopping time is so important and we can use it to describe a generic situation in which to execute or revise a stopping decision. Consistency will be a particularly important issue here, of course.

Given this concept of subgames, we study subgame-perfect equilibria in \emph{mixed strategies} for stochastic timing games. In many timing games in continuous time there exist no equilibria in pure strategies, basically due to the fact that there is no ``next'' period, onto which one can put a threat, for instance. This fact tends to be overlooked in part of the more applied literature and we aim to provide a sound foundation for equilibria that have already become standard in the strategic real option literature (e.g., the papers listed below). The primary issue is conceptual: to consider the players' options properly. This then entails also some technical issues that one has to address. Another important reason to introduce mixed strategies is that pure strategy equilibria in symmetric games~-- if they exist~-- often involve asymmetric payoffs that just depend on the respective roles taken by the players. Determining the roles can be regarded as an additional strategic problem to solve before the timing game starts, which may be avoided by allowing mixed strategies.

Our basic concept of mixed strategies are distributions over time that may depend on the state of the world, but only through the available information about uncertainty: they will be adapted processes. Together with our requirement of time consistency we obtain objects that correspond to behaviour strategies in discrete-time models as far as the analogy can reach~-- since our decision nodes cannot be well ordered in continuous time and it is not possible to represent distributions over time by specifying conditional probabilities for each point in their support separately.

It is well known that coordination is very important in many timing games, in particular when there is a preemption incentive, like in typical strategic real option models. Often some ad hoc modification of the actual model is used~-- like tie breaking by coin tosses~-- to circumvent such issues and to obtain any equilibria.\footnote{
See, e.g., \cite{HoppeLehmann-Grube05}.
}
However, the risk of simultaneous stopping (resp.\ investment) is a priori one of the key aspects of such situations. The most notable approach to enable just sufficient coordination by extending strategy spaces has certainly been that of \cite{FudenbergTirole85}, which also aims to capture limits from discrete-time approximations of the continuous-time game. We generalize their concept to a stochastic framework, stressing the role of limit outcomes.\footnote{
\cite{Thijssenetal12} take a different route to adapt the approach of Fudenberg and Tirole to a stochastic setting. First, they use \emph{unconditional} strategies, which do not depend on whether the respective other player has already stopped. Second, they force strategies and outcome distributions to be the same by imposing a joint restriction on feasible pairs of strategies. The point of \cite{FudenbergTirole85} is exactly that outcome probabilities and strategies need not be the same if one argues with continuous-time limits from discrete time. Our strategies are conditional on the stopping history and we allow players to choose any strategy from their individually feasible set and then determine the resulting outcome distributions.
} 
To accomodate standard asymmetric or stochastic models we have to drop some regularity requirements that Fudenberg and Tirole use at exactly the most critical point, which has been a somewhat distracting feature. We propose a more fundamental solution based on limit outcomes, which however still requires a careful introduction.

In a stochastic setting optimality is much less obvious than in a deterministic one, since the game and the local incentives do not proceed quite as linearly. Therefore we stress the role of optimal stopping and develop complete formal proofs for the subgame-perfect equilibria we propose. Specifically, we provide some general characterization of preemptive equilibria and we illustrate its application by some typical models that allow explicit solutions.

Strategic timing problems appear in an abundance of contexts, and in particular in economics. Hence there is a vast literature on this classical topic and we only name a few works that are most related to ours for certain reasons. 

On the one hand there is the literature on mainly deterministic timing problems in continuous time that is inspired by a wide range of applications, such as preemption models in economics (e.g., \cite{FudenbergTirole85,HendricksWilson92}) or wars of attrition in biology or economics as well (e.g., \cite{Hendricksetal88}). These classes of models are quite stylized with a systematic first or second-mover advantage. Then, without uncertainty, the game proceeds indeed very linearly due to perfect foresight. Some complications arise when the incentives may vary. \cite{Larakietal05} consider general deterministic \nbd{N}player games with payoffs that are just continuous functions of time (for given identities of first-movers). They prove that there do always exist \nbd{\varepsilon}equilibria, but not necessarily exact equilibria.

On the other hand, as we emphasize uncertainty, the literature on Dynkin games with a large tradition in mathematics also has to be named here. Classically, however, these are two-person, \emph{zero-sum} timing games, and the classical question is the existence of an equilibrium (saddle) point, called value, under varying conditions. We here just refer to the more recent work by \cite{TouziVieille02}, since their payoff processes are very general and~-- more importantly~-- since they introduce another concept of mixed strategies (but without consideration of subgames). \cite{TouziVieille02} prove that many more Dynkin games have a value if one allows for such mixed strategies.

Quite recently the two strands began to merge by considering stochastic timing games with non-zero-sum payoffs. \cite{HamadeneZhang10}, for instance, prove existence of Nash equilibrium for 2-player games with a general second-mover advantage.\footnote{
See also \cite{HamadeneHassani14} for an extension to \(N\) players using a similar approach. \cite{LarakiSolan13} make less assumptions concerning the incentives in a 2-player game. Consequently, even allowing for mixed strategies, they can only prove existence of \nbd{\varepsilon}equilibria.
}.

The type of application we are mainly addressing is strategic investment under uncertainty. An early model that we will have a closer look at is the one of \cite{Weeds02} and similar ones of \cite{PawlinaKort06} or \cite{MasonWeeds10} and followers. We propose strategies that do support the equilibrium outcomes described in these papers.

We begin by defining the stochastic timing game and our notion of subgames and mixed strategies in Section \ref{sec:game}. In Section \ref{sec:extended} we derive equilibria in extended mixed strategies for preemption games, i.e., with local first-mover advantages. Our concepts and general results are illustrated by several applications in Section \ref{sec:exm}. An appendix collects the proofs and some technical results.

\section{The notion of subgame in continuous-time stopping games}\label{sec:game}

We consider a timing game between two players in continuous time  under uncertainty.

Let $\bigl(\,\Omega, \F, (\F_t)_{t \ge 0}, P\,\bigr)$ be a filtered probability space. Let six adapted, right-continuous processes $L^i,F^i,$ and $M^i$ for two players $i=1,2$  be given. They correspond to the  leader's, follower's, or simultaneous stopping payoff, resp., in the continuous-time stochastic stopping game that we are going to set up now. The index $i$ is added to allow for asymmetric payoffs.

A pure strategy is a stopping time for a player. Recall that a stopping time $\tau$ is a random time such that the event ``stop before time $t$'' is known at time $t$, or $\left\{\tau \le t \right\} \in \F_t \;\forall t\geq 0$.
Denote the set of all stopping times or pure strategies by $\T$.

Payoffs depend on the stopping actions of both players.  Whenever one of the players stops, the game ends.
If player \(i\in\{1,2\}\) is the sole first to stop at the stopping time \(\tau\), he obtains the (``leader's'') payoff \(L^i_\tau\). The opponent \(j\in\{1,2\}\setminus i\) gets  the (``follower's'')  payoff \(F^j_\tau\).\footnote{
$F^j_\tau$ may incorporate the value of the continuation problem of the remaining player \(j\) if there still is a payoff-relevant stopping decision to make. Then, if some player \(i\) has already stopped, a new situation arises for which separate strategies have to be formed that determine the eventual stopping of the follower \(j\). We concentrate only on histories in which no one has stopped, yet, as it is customary in the literature on stopping games.
} 
If both players stop simultaneously, each player \(i\) obtains \(M^i_\tau\). Suppose players \(i\) and \(j\) plan to stop at the stopping times \(\tau_i\) and \(\tau_j\), respectively. Then the game ends at the stopping time \(\tau=\tau_i\wedge\tau_j\) and the payoff to player \(i\) at time 0 is
\begin{equation}\label{Vi_pure}
\pi_i\left(\tau_i,\tau_j\right):=E\bigl[L^i_{\tau_i}\indi{\tau_i<\tau_j}+F^i_{\tau_j}\indi{\tau_j<\tau_i}+M^i_{\tau_i}\indi{\tau_i=\tau_j}\bigr].
\end{equation}

\begin{definition}
A timing game $\Gamma$ is a tuple
$$\Bigl( \bigl(\,\Omega, \F, (\F_t)_{t \ge 0}, P\,\bigr), \T\times \T, \bigl(\,L^i,F^i, M^i\,\bigr)_{i=1,2}, \bigl(\pi_i\bigr)_{i=1,2} \Bigr)$$
consisting of a filtered probability space $\bigl(\,\Omega, \F, (\F_t)_{t \ge 0}, P\,\bigr)$,
stopping times $\T$ as pure strategies, adapted processes
$\bigl(\,L^i,F^i, M^i\,\bigr)_{i=1,2}$, and payoffs $\pi_i$ as defined in (\ref{Vi_pure}).
\end{definition}

Obviously, equilibria will be based on solving optimal stopping problems involving the three underlying payoff processes. We need to make some standard regularity assumptions in order to have well defined problems in the following.

\begin{assumption}\label{asm:payoffs}
\par\noindent
\begin{enumerate}
\item
The filtration \(\filt{F}:=\left(\F_t\right)_{t\geq 0}\) satisfies the usual conditions (i.e., \(\filt{F}\) is right-continuous and complete).
\item
The processes \(L^i\), \(F^i\) and \(M^i\), \(i\in\{1,2\}\), are adapted, right-continuous (a.s.) and of class {\rm (D)}, \(M^i\) having an extension with \(E[\lvert M^i_\infty\rvert]<\infty\).
\item
\(F^i\geq M^i\) (a.s., for any \(t\in\R_+\)).
\end{enumerate}
\end{assumption}

\begin{remark}
\par\noindent
\begin{enumerate}
\item\label{rem:XclassD} 
A measurable process \(X\) is of class {\rm (D)} if the family \(\{X_\tau:\tau<\infty\) a.s.\ a stopping time\(\}\) is uniformly integrable, so that the family is bounded in \(L^1(P)\) and pointwise convergence of \(X\) at a stopping time implies convergence in \(L^1(P)\) as well. This is a mild regularity condition implied, e.g., by either \(E[\sup_t\abs{X_t}]<\infty\) or \(\sup_\tau E[\abs{X_\tau}^p]<\infty\) for some \(p>1\). We may equivalently define any extension \(X_\infty\in L^1(P)\) and consider \emph{all} stopping times (possibly taking the value \(\infty\)) in the previous set; cf.\ Lemma \ref{lem:classD}.

\item
It depends on the model whether there is a natural payoff if both players ``never stop'', which may be some limit of \(M^i\) or of \(L^i\). In the latter case we simply set \(M^i_\infty:=L^i_\infty\) and work with \(M^i_\infty\) for a unified payoff notation. For convenience we also define
\begin{equation*}
F^i_\infty:=M^i_\infty.
\end{equation*}

\item The assumption $F^i \ge M^i$ is quite natural if the follower still has a stopping decision to make, such that $F^i$ is the corresponding value function where we include simultaneous stopping as an option for the ``follower''. On a more abstract level the condition means that we are focussing on competitive models without a strict benefit from coordinating.
\end{enumerate}
\end{remark}

\subsection{Subgames}\label{sec:subgame}

Stopping times play a twofold role in our setup. On the one hand, they are players' pure strategies as we have seen above, because they are exactly the feasible plans to stop given the dynamic information \(\filt{F}=\left(\F_t\right)_{t\geq 0}\). On the other hand, we now argue that they encode the starting point of a subgame in our continuous-time framework.

In discrete time and discrete state spaces,  subgames start at a certain node in the game tree; the node needs to be a singleton information set. Such a node can be naturally described by a (time) level $t$ in the tree and by an identifiable set of states of the world (which node $k$ at time $t$ is the case). The random time
$\tau= t$ if we are in node $k$ at time $t$ and $\tau=\infty$ otherwise is then a stopping time that characterizes that node. In continuous time it is not enough to identify subgames by considering times $t\in\R_+$ and events which are measurable at those times, because stopping times are much richer.\footnote{
Similarly the \nbd{\sigma}fields \(\F_\tau\) associated with stopping times are a more general concept of a ``past'' than those for constant times, \(\F_t\).
}
A typical example for a stopping time in standard stochastic models is the first time a Brownian motion exceeds a certain fixed value $b>0$. This ``first passage time'' (or hitting time) has a continuous distribution with full support on $(0,\infty)$.\footnote{
The first passage time has the density \(b(2\pi t^3)^{-\frac12}e^{-b^2/2t}>0\) on \((0,\infty)\); see, e.g., \cite{RevuzYor}, Section II.3.
}
When a strategy is to stop at such a hitting time, we also have to know for subgame perfection what will happen if we are at the hitting time, but stopping does not occur (i.e., off the path of play). If we only had plans for deterministic times $t$, we would have to assemble one for the hitting time out of uncountably many probability zero events, which would not give a well-defined object in general.

Thus we use the notion of stopping time not only as feasible rules when to stop, but also as feasible decision nodes to define subgames in stochastic timing games, where players' plans may be revised.

\begin{definition}
Let $\vartheta\in\T$ be a stopping time. Let $\Gamma$ be a timing game. The subgame $\Gamma^\vartheta$ starting at the stopping time $\vartheta$ is the tuple
$$\Bigl( \bigl(\,\Omega, \F, (\F_t)_{t \ge 0}, P\,\bigr), \T_\vartheta\times \T_\vartheta, \bigl(\,L^i,F^i, M^i\,\bigr)_{i=1,2}, \bigl(\pi^\vartheta_i\bigr)_{i=1,2} \Bigr)$$
consisting of the filtered probability space $\bigl(\,\Omega, \F, (\F_t)_{t \ge 0}, P\,\bigr)$,
stopping times greater equal $\vartheta$
$$\T_\vartheta := \left\{ \tau \in \T : \tau \ge \vartheta \right\}$$ as pure strategies, adapted processes
$\bigl(\,L^i,F^i, M^i\,\bigr)_{i=1,2}$ and conditional payoffs
$$\pi^\vartheta_i\left(\tau_i,\tau_j\right)=E\bigl[L^i_{\tau_i}\indi{\tau_i<\tau_j}+F^i_{\tau_j}\indi{\tau_j<\tau_i}+M^i_{\tau_i}\indi{\tau_i=\tau_j}\bigm|\F_\vartheta \bigr],\quad\tau_i,\tau_j\in\T_\vartheta.$$
\end{definition}

Our definition of subgames can also be seen as an analogy to the general approach to stopping problems for a single decision maker.\footnote{
See, e.g., \cite{ElKaroui81} for the general theory of optimal stopping and the concept of the \emph{Snell envelope}.
}
In general the solution of an optimal stopping problem can be represented as a consistent collection of contingent plans for starting from \emph{any stopping time}. Similarly our approach allows us to speak meaningfully of best replies in subgames and to define subgame-perfect equilibria.

\subsubsection{Mixed strategies in subgames}\label{sec:strateql}

Many timing games have no equilibria in pure strategies.\footnote{
See also \cite{HendricksWilson92} on (non-)existence of equilibria in deterministic preemption games.
}
We thus introduce mixed strategies, following the approach by \cite{FudenbergTirole85} but generalizing it to our stochastic setting. In principle we consider a mixed strategy as a (random) distribution function over time.\footnote{
An alternative approach is to randomize over stopping times before the game starts. \cite{TouziVieille02} show that the two approaches are payoff-equivalent. They do not consider, however, any notion of subgame (perfection) or further extensions as we do.
}
However, as we are aiming for subgame-perfection, we need players to make plans also for times at which stopping will already have occurred with probability one according to the chosen distributions, i.e., which are not expected to be reached. Therefore we first introduce distributions over remaining time as complete plans of action for every subgame and then aggregate them as strategies for the whole timing game. Imposing natural consistency conditions avoids contradictions and ensures well-defined outcomes.

\begin{definition}\label{def:alpha}
Fix a stopping time \(\vartheta\in\T\). An \emph{extended mixed strategy} for player \(i\in\{1,2\}\) in the subgame \(\Gamma^\vartheta\) starting at \(\vartheta\), also called \emph{\nbd{\vartheta}strategy}, is a pair of processes \(\bigl(G^\vartheta_i,\alpha^\vartheta_i\bigr)\) taking values in \([0,1]\), respectively, with the following properties. 
\begin{enumerate}
\item
\(G^\vartheta_i\) is adapted. It is right-continuous and non-decreasing with \(G^\vartheta_{i}(s)=0\) for all \(s<\vartheta\), a.s. 

\item
\(\alpha^\vartheta_i\) is progressively measurable. It is right-continuous where its value is in \((0,1)\), a.s.

\item
\begin{equation*}
\alpha^\vartheta_i(t)>0\Rightarrow G_i^\vartheta(t)=1\qquad\text{for all }t\geq 0\text{, a.s.}
\end{equation*}
\end{enumerate}
We further define \(G^\vartheta_{i}(0-)\equiv 0\), \(G^\vartheta_{i}(\infty)\equiv 1\) and \(\alpha^\vartheta_i(\infty)\equiv 1\) for every extended mixed strategy.
\end{definition}

\begin{remark}
\par\noindent
\begin{enumerate}
\item
As in \cite{FudenbergTirole85}, the extensions \(\alpha^\vartheta_i\) are used as a coordination instrument where the distribution functions \(G^\vartheta_i\) jump to 1, and outcomes will be determined below by a similar limit argument, relying on right-continuity of the former. \(\alpha^\vartheta_i=1\) means definite immediate stopping by player \(i\), whence we do not need a right-hand limit to identify the outcome. In order to accomodate asymmetric models we cannot generally require right-continuity where \(\alpha^\vartheta_i=0\), either; see Section \ref{subsec:issupreem}.
\item
Progressive measurability ensures enough structure in the time domain such that \(\alpha^\vartheta_i(\tau)\) will be \nbd{\F_\tau}measurable for any \(\tau\in\T\).\footnote{
Formally, the mapping \(\alpha^\vartheta_i:\Omega\times[0,t]\to\R\), \((\omega,s)\mapsto\alpha^\vartheta_i(\omega,s)\) must be \(\F_t\otimes\B([0,t])\)-measurable for any \(t\in\R_+\). It is a stronger condition than (i.e., it implies) adaptedness, but weaker than (i.e., implied by) optionality, which we have for \(G^\vartheta_i\) by right-continuity. However, given Definition \ref{def:outcome} of outcome probabilities, we can also assume \(\alpha^\vartheta_i\) to be optional without loss by using its unique optional projection \({}^o\alpha^\vartheta_i\) instead, which agrees with \(\alpha^\vartheta_i\) a.s.\ at any \(\tau\in\T\). \({}^o\alpha^\vartheta_i\) is also right-continuous where \(\alpha^\vartheta_i\) is so at any \(\tau\in\T\), see Lemma \ref{lem:oalpharc}.
}   
The extensions at \(t=0\) and \(t=\infty\) are defined only for notational convenience in the definition of payoffs. 
\end{enumerate}
\end{remark}

Now we aggregate strategies for all single subgames to a strategy for the whole timing game. It can well be that two stopping times coincide with a certain positive probability. The strategies are required to give unique and well-defined prescriptions in that case.

\begin{definition}\label{def:strat}
An \emph{extended mixed strategy} for player \(i\in\{1,2\}\) for the timing game \(\Gamma\) is a family of \nbd{\vartheta}strategies \(\bigl(G_i,\alpha_i\bigr)=\bigl(G^\vartheta_i,\alpha^\vartheta_i\bigr)_{\vartheta\in\T}\) satisfying the consistency condition \(G^\vartheta_i=G^{\vartheta'}_i\) and \(\alpha^\vartheta_i=\alpha^{\vartheta'}_i\) a.s.\ on the event \(\{\vartheta=\vartheta'\}\) for all stopping times \(\vartheta,\vartheta'\in\T\).\footnote{
This consistency condition is also implied by the time consistency introduced in Definition \ref{def:TC_extended}, because there we require \(G^\vartheta_i=G^{\vartheta\wedge\vartheta'}_i\) on \(\{\vartheta=\vartheta\wedge\vartheta'\}\) and \(G^{\vartheta'}_i=G^{\vartheta\wedge\vartheta'}_i\) on \(\{\vartheta'=\vartheta\wedge\vartheta'\}\); similarly for \(\alpha^\cdot_i\).
}
We denote the set of strategies for each player by \(\S\).
\end{definition}

\subsubsection{Payoffs}

We have to generalize the preliminary payoffs \eqref{Vi_pure} to subgames and (extended) mixed strategies.

\begin{definition}\label{def:payoffs_extended}
Given two extended mixed strategies \(\bigl(G^\vartheta_i,\alpha^\vartheta_i\bigr)\), \(\bigl(G^\vartheta_j,\alpha^\vartheta_j\bigr)\), \(i,j\in\{1,2\}\), \(i\not=j\),  the \emph{payoff} of player \(i\) in the subgame starting at \(\vartheta\in\T\) is
\begin{align*}
V^\vartheta_{i}\bigl(G^\vartheta_{i},\alpha^\vartheta_i,G^\vartheta_{j},\alpha^\vartheta_j\bigr):=E\biggl[&\int_{[0,\hat\tau^\vartheta)}\bigl(1-G^\vartheta_{j}(s)\bigr)L^i_s\,dG^\vartheta_{i}(s)\nonumber\\
+{}&\int_{[0,\hat\tau^\vartheta)}\bigl(1-G^\vartheta_{i}(s)\bigr)F^i_s\,dG^\vartheta_{j}(s)\nonumber\\
+{}&\sum_{s\in[0,\hat\tau^\vartheta)}\Delta G^\vartheta_{i}(s)\Delta G^\vartheta_{j}(s)M^i_s\\
+{}&\lambda^\vartheta_{L,i}L^i_{\hat\tau^\vartheta}+\lambda^\vartheta_{L,j}F^i_{\hat\tau^\vartheta}+\lambda^\vartheta_{M}M^i_{\hat\tau^\vartheta}\biggm|\F_\vartheta\biggr],
\end{align*}
with \(\hat\tau^\vartheta\), \(\lambda^\vartheta_{L,i}\), \(\lambda^\vartheta_{L,j}\) and \(\lambda^\vartheta_{M}\) as in Definition \ref{def:outcome} below.
\end{definition}

Lemma \ref{lem:LdG} ensures not only that the pathwise integrals (which include possible jumps of the \emph{right-continuous} integrators at 0, since \(i\) can indeed become leader/follower from an initial jump of \(G^\vartheta_{i}\)/\(G^\vartheta_{j}\), resp.) are well defined under Assumption \ref{asm:payoffs}, but also that the payoffs are bounded in \(L^1(P)\), uniformly across all feasible strategies. 

If the players do not use the extensions (i.e., \(\alpha^\vartheta_1=\alpha^\vartheta_2\equiv 0\) on \([\vartheta,\infty)\)), then \(\hat\tau^\vartheta=\infty\), \(\lambda^\vartheta_{L,1}=\lambda^\vartheta_{L,2}=0\) and \(\lambda^\vartheta_{M}=\Delta G^\vartheta_{1}(\infty)\Delta G^\vartheta_{2}(\infty)\), so we have the same payoffs as in the analogous model with only the mixed strategies \(G^\vartheta_{1}\), \(G^\vartheta_{2}\). Precisely, the ``terminal'' payoffs then are \((1-G^\vartheta_{i}(\infty-))(1-G^\vartheta_{j}(\infty-))M^i_\infty\) because we have defined \(G^\vartheta_i(\infty)=\alpha^\vartheta_i(\infty)=1\). 

\(\hat\tau^\vartheta\) denotes the first time any of the extensions \(\alpha^\vartheta_\cdot\) is positive. If \(\hat\tau^\vartheta<\infty\), then \(\alpha^\vartheta_1\) and \(\alpha^\vartheta_2\) determine final outcome probabilities by the limit interpretation of \cite{FudenbergTirole85}, but with our much weaker regularity assumptions needed to accomodate asymmetric or stochastic games. In particular, Fudenberg and Tirole require their \(\alpha_i(\cdot)\) to be right-differentiable to apply a Taylor-expansion at the decisive point, when the \(\alpha_i(\cdot)\) start to be positive. One cannot hope for such smoothness in stochastic models. With less restrictions, we need to take more care that outcomes are always well defined and consistent with the limit argument whenever possible. Once we have done so, for later applications one can rely on general results about equilibria in our framework presented in Section \ref{sec:extended}, as it is also illustrated in Section \ref{sec:exm}.

Recall from \cite{FudenbergTirole85} that the \(\alpha^\vartheta_i(\cdot)\) are interpreted as the limits of stage stopping probabilities in a repeated stopping game where one lets the period length vanish. A right-continuous limit \(\alpha^\vartheta_i(\cdot)\) means that the sequences of stage stopping probabilities are basically constant on small fixed time intervals. The corresponding limit of outcome probabilities~-- of who stops first: player 1, 2 or both~-- is then that of an infinitely repeated stopping game with the given constant stage probabilities, determined as follows. Define the functions \(\mu_L\) and \(\mu_M\) from \([0,1]^2\setminus (0,0)\) to \([0,1]\) by
\begin{align*}
\mu_L(x,y):=\frac{x(1-y)}{x+y-xy}\qquad\text{and}\qquad\mu_M(x,y):=\frac{xy}{x+y-xy}.
\end{align*}
\(\mu_L(a_i,a_j)\) is the probability that player \(i\) stops first in an infinitely repeated stopping game where \(i\) plays constant stage stopping probabilities \(a_i\) and player \(j\) plays constant stage probabilities \(a_j\). \(\mu_M(a_i,a_j)\) is the probability of simultaneous stopping and \(1-\mu_L(a_i,a_j)-\mu_M(a_i,a_j)=\mu_L(a_j,a_i)\) that of player \(j\) stopping first. Only \(\mu_M\) admits a continuous extension at the origin, \(\mu_M(0,0):=0\), but \(\mu_L\) does not, which requires careful treatment.

\begin{definition}\label{def:outcome}
Given \(\vartheta\in\T\) and a pair of extended mixed strategies \(\bigl(G^\vartheta_1,\alpha^\vartheta_1\bigr)\) and \(\bigl(G^\vartheta_2,\alpha^\vartheta_2\bigr)\), the \emph{outcome probabilities} \(\lambda^\vartheta_{L,1}\), \(\lambda^\vartheta_{L,2}\) and \(\lambda^\vartheta_M\) at
\begin{equation*}
\hat\tau^\vartheta:=\inf\{t\geq\vartheta\mid\alpha^\vartheta_1(t)+\alpha^\vartheta_2(t)>0\}
\end{equation*}
are defined as follows.\footnote{
There are no conditional expectations in the definition, which one might expect as we are taking limits of ``future'' outcome probabilities that are themselves random. However, we have pointwise right-hand limits and boundedness where \(\alpha^\vartheta_i\in(0,1)\) and thus convergence  of expectations when we apply the limit argument. Even if the \(\liminf(\cdot)\) and \(\limsup(\cdot)\) in the last case in the definition differ, the latter are progressively measurable processes by Theorem IV.33 (c) in \cite{DellacherieMeyer78} and hence \nbd{\F_\tau}measurable for any \(\tau\in\T\).
}
Let \(i,j\in\{1,2\}\), \(i\not=j\). 

\noindent
If \(\hat\tau^\vartheta<\hat\tau^\vartheta_j:=\inf\{t\geq\vartheta\mid\alpha^\vartheta_j(t)>0\}\), then
\begin{align*}
\lambda^\vartheta_{L,i}:={}&\bigl(1-G_i^\vartheta(\hat\tau^\vartheta-)\bigr)\bigl(1-G_j^\vartheta(\hat\tau^\vartheta)\bigr),\\
\\
\lambda^\vartheta_M:={}&\bigl(1-G_i^\vartheta(\hat\tau^\vartheta-)\bigr)\alpha^\vartheta_i(\hat\tau^\vartheta)\Delta G_j^\vartheta(\hat\tau^\vartheta).
\end{align*}
If \(\hat\tau^\vartheta<\hat\tau^\vartheta_i:=\inf\{t\geq\vartheta\mid\alpha^\vartheta_i(t)>0\}\), then
\begin{align*}
\lambda^\vartheta_{L,i}:={}&\bigl(1-G_j^\vartheta(\hat\tau^\vartheta-)\bigr)\bigl(1-\alpha_j(\hat\tau^\vartheta)\bigr)\Delta G_i^\vartheta(\hat\tau^\vartheta),\\
\\
\lambda^\vartheta_M:={}&\bigl(1-G_j^\vartheta(\hat\tau^\vartheta-)\bigr)\alpha^\vartheta_j(\hat\tau^\vartheta)\Delta G_i^\vartheta(\hat\tau^\vartheta).
\end{align*}
If \(\hat\tau^\vartheta=\hat\tau^\vartheta_1=\hat\tau^\vartheta_2\) and \(\alpha^\vartheta_1(\hat\tau^\vartheta)\vee\alpha^\vartheta_2(\hat\tau^\vartheta)=1\) or \(\alpha^\vartheta_1(\hat\tau^\vartheta)\wedge\alpha^\vartheta_2(\hat\tau^\vartheta)>0\), resp., then
\begin{align*}
\lambda^\vartheta_{L,i}:={}&\bigl(1-G_i^\vartheta(\hat\tau^\vartheta-)\bigr)\bigl(1-G_j^\vartheta(\hat\tau^\vartheta-)\bigr)\mu_L(\alpha^\vartheta_i(\hat\tau^\vartheta),\alpha^\vartheta_j(\hat\tau^\vartheta)),\\
\\
\lambda^\vartheta_M:={}&\bigl(1-G_i^\vartheta(\hat\tau^\vartheta-)\bigr)\bigl(1-G_j^\vartheta(\hat\tau^\vartheta-)\bigr)\mu_M(\alpha^\vartheta_1(\hat\tau^\vartheta),\alpha^\vartheta_2(\hat\tau^\vartheta)).
\end{align*}
If \(\hat\tau^\vartheta=\hat\tau^\vartheta_1=\hat\tau^\vartheta_2\), \(\alpha^\vartheta_1(\hat\tau^\vartheta)\vee\alpha^\vartheta_2(\hat\tau^\vartheta)<1\) and \(\alpha^\vartheta_1(\hat\tau^\vartheta)\wedge\alpha^\vartheta_2(\hat\tau^\vartheta)=0\), then
\begin{align*}
\lambda^\vartheta_{L,i}:={}&\bigl(1-G_i^\vartheta(\hat\tau^\vartheta-)\bigr)\bigl(1-G_j^\vartheta(\hat\tau^\vartheta-)\bigr)\bigl(1-\alpha_j^\vartheta(\hat\tau^\vartheta)\bigr)\\
&\cdot\biggl(\alpha_i^\vartheta(\hat\tau^\vartheta)+\bigl(1-\alpha_i^\vartheta(\hat\tau^\vartheta)\bigr)\,\frac{1}{2}\!
\begin{aligned}[t]\biggl\{&\liminf_{\underset{\alpha^\vartheta_i(t)+\alpha^\vartheta_j(t)>0}{t\searrow\hat\tau^\vartheta}}\mu_L(\alpha^\vartheta_i(t),\alpha^\vartheta_j(t))\biggr.\\
\biggl.+&\limsup_{\underset{\alpha^\vartheta_i(t)+\alpha^\vartheta_j(t)>0}{t\searrow\hat\tau^\vartheta}}\mu_L(\alpha^\vartheta_i(t),\alpha^\vartheta_j(t))\biggr\}\biggr),\end{aligned}\\
\\
\lambda^\vartheta_M:={}&\bigl(1-G_i^\vartheta(\hat\tau^\vartheta-)\bigr)\bigl(1-G_j^\vartheta(\hat\tau^\vartheta-)\bigr)-\lambda^\vartheta_{L,i}-\lambda^\vartheta_{L,j}\\
={}&\bigl(1-G_i^\vartheta(\hat\tau^\vartheta-)\bigr)\bigl(1-G_j^\vartheta(\hat\tau^\vartheta-)\bigr)\bigl(1-\alpha_i^\vartheta(\hat\tau^\vartheta)\bigr)\bigl(1-\alpha_j^\vartheta(\hat\tau^\vartheta)\bigr)\\
&\cdot\mu_M(\alpha^\vartheta_1(\hat\tau^\vartheta+),\alpha^\vartheta_2(\hat\tau^\vartheta+))\qquad\qquad\text{if }\alpha^\vartheta_1(\hat\tau^\vartheta+)\text{ and }\alpha^\vartheta_2(\hat\tau^\vartheta+)\text{ exist}.
\end{align*}
\end{definition}

\begin{remark}
\noindent
\begin{enumerate}
\item
\(\lambda^\vartheta_M\) is the probability of simultaneous stopping at \(\hat\tau^\vartheta\), while \(\lambda^\vartheta_{L,i}\) is the probability of player \(i\) becoming the leader, i.e., that of player \(j\) becoming follower. It holds that \(\lambda^\vartheta_M+\lambda^\vartheta_{L,i}+\lambda^\vartheta_{L,j}=\bigl(1-G_i^\vartheta(\hat\tau^\vartheta-)\bigr)\bigl(1-G_j^\vartheta(\hat\tau^\vartheta-)\bigr)\). Dividing by \(\bigl(1-G_i^\vartheta(\hat\tau^\vartheta-)\bigr)\bigl(1-G_j^\vartheta(\hat\tau^\vartheta-)\bigr)\) where feasible yields the corresponding conditional probabilities, which will have to satisfy time consistency below. 

\item
In the definition we first consider the two cases when one player plays \(\alpha^\vartheta_\cdot\) positive against an isolated mass point \(\Delta G^\vartheta_\cdot\) of the other. In the third case either some \(\alpha^\vartheta_\cdot=1\) and we do not need a limit, or the latter is well-behaved by right-continuity. In the last case we have a limit for simultaneous stopping as soon as \(\alpha^\vartheta_1(\hat\tau^\vartheta+)\) and \(\alpha^\vartheta_2(\hat\tau^\vartheta+)\) exist because \(\mu_M\) is continuous also at 0. However, there might be no limit of \(\mu_L\) if \(\alpha^\vartheta_1(\hat\tau^\vartheta)=\alpha^\vartheta_2(\hat\tau^\vartheta)=0\) even when both \(\alpha^\vartheta_\cdot\) are continuous.\footnote{
If \(\alpha^\vartheta_1(\hat\tau^\vartheta)\vee\alpha^\vartheta_2(\hat\tau^\vartheta)<1\) and \(\alpha^\vartheta_1(\hat\tau^\vartheta+)\), \(\alpha^\vartheta_2(\hat\tau^\vartheta+)\) exist, then the limit \emph{only} may not exist if \(\alpha^\vartheta_1(\hat\tau^\vartheta+)=\alpha^\vartheta_2(\hat\tau^\vartheta+)=0\), i.e., if \(\alpha^\vartheta_1(\hat\tau^\vartheta)=\alpha^\vartheta_2(\hat\tau^\vartheta)=0\): if \(\alpha^\vartheta_i(\hat\tau^\vartheta+)>0\), the limit of \(\mu_L\) is determined by continuity. If the limit in a potential equilibrium does not exist, both players will be indifferent about the roles; see Lemma \ref{lem:limit}.
}  

Here we differ from \cite{FudenbergTirole85} who ask for right-differentiability and a positive derivative to apply a Taylor expansion, which is a too strong requirement for asymmetric or stochastic models; cf.\ Section \ref{sec:extended}. Taking instead the symmetric combination of \(\liminf\) and \(\limsup\) ensures consistency whenever the limit exists, independence of the players' names, and that \(\lambda^\vartheta_M\) coincides with its associated limit whenever the latter exists. Furthermore, our solution provides no incentives for the players to create ambiguity about the limit of \(\mu_L\) by their choice of strategies (i.e., to exploit the rule of the last case above): in Section \ref{sec:extended} we show that if there exists a best reply to an extended mixed strategy, then there is a pure one in general.
\end{enumerate}
\end{remark}

\subsubsection{Time consistency and subgame-perfect equilibrium}

Our definition of strategies already included a consistency condition to avoid contradictions in view of the high level of redundancy, given the potentially enormous number of stopping times, i.e., subgames differing only on small events. For a subgame-perfect equilibrium we further require time consistency in the form of Bayes' law, resp.\ that conditional stopping probabilities agree.

\begin{definition}\label{def:TC_extended}
An extended mixed strategy \(\bigl(G_i,\alpha_i\bigr)\) is \emph{time-consistent} if for all \(\vartheta\leq\vartheta'\in\T\)
\begin{flalign*}
&& \vartheta'\leq t\in\R_+ &\Rightarrow\ G_i^\vartheta(t)=G_i^\vartheta(\vartheta'-)+\bigl(1-G_i^\vartheta(\vartheta'-)\bigr)G_i^{\vartheta'}(t)\quad\text{a.s.} && \\
&\text{and} \\
&& \vartheta'\leq\tau\in\T &\Rightarrow\ \alpha^\vartheta_i(\tau)=\alpha^{\vartheta'}_i(\tau)\quad\text{a.s.} &&
\end{flalign*}
\end{definition}

The equilibrium concept is then natural.

\begin{definition}\label{def:SPE_extended}
A \emph{subgame-perfect equilibrium} for the timing game is a pair \(\bigl(G_1,\alpha_1\bigr)\), \(\bigl(G_2,\alpha_2\bigr)\) of time-consistent extended mixed strategies such that for all \(\vartheta\in\T\), \(i,j\in\{1,2\}\), \(i\not=j\), and extended mixed strategies \(\bigl(G_a^\vartheta,\alpha^\vartheta_a\bigr)\)
\begin{equation*}
V_i^\vartheta(G_i^\vartheta,\alpha^\vartheta_i,G_j^\vartheta,\alpha^\vartheta_j)\geq V_i^\vartheta(G_a^\vartheta,\alpha^\vartheta_a,G_j^\vartheta,\alpha^\vartheta_j)\quad\text{a.s.},
\end{equation*}
i.e., such that every pair \(\bigl(G^\vartheta_1,\alpha^\vartheta_1\bigr)\), \(\bigl(G^\vartheta_2,\alpha^\vartheta_2\bigr)\) is an \emph{equilibrium} in the subgame at \(\vartheta\in\T\), respectively.
\end{definition}

\section{The role of extended mixed strategies in equilibrium}\label{sec:extended}

In this section we analyze the payoffs that can result from extended mixed strategies and show that equilibrium conditions have strong implications for the relevant choices of \(\alpha^\vartheta_i\). In particular we establish a quite general characterization of subgame-perfect equilibria in games of preemption type, which addresses the issues with equilibria proposed in the literature and which we will use intensely in the applications in Section \ref{sec:exm}.

Whenever some \(\alpha^\cdot_j\) is positive, player \(i\) can secure at least the conditional payoff \(F^i\) because the game definitely ends. More specifically, suppose \(\vartheta=\hat\tau^\vartheta_j=\inf\{t\geq\vartheta\mid\alpha^\vartheta_j(t)>0\}\). Then \(G^\vartheta_{j}(\vartheta)=1\) and player \(i\)'s payoff from any strategy \(\bigl(G^\vartheta_{i},\alpha^\vartheta_i\bigr)\) will be at most that from stopping at infinity or that from stopping immediately, i.e., \(V^\vartheta_{i}\bigl(G^\vartheta_{i},\alpha^\vartheta_i,G^\vartheta_{j},\alpha^\vartheta_j\bigr)\leq\max\{F^i_\vartheta,\alpha^\vartheta_j(\vartheta)M^i_\vartheta+\bigl(1-\alpha^\vartheta_j(\vartheta)\bigr)L^i_\vartheta\}\).\footnote{
The only case where this may not be so easy to check is when the outcome probabilities \(\lambda^\vartheta_{L,i}\) and \(\lambda^\vartheta_M\) involve non-trivial limits due to \(\vartheta=\hat\tau^\vartheta_1=\hat\tau^\vartheta_2\), \(\alpha^\vartheta_1(\vartheta)\wedge\alpha^\vartheta_2(\vartheta)=0\) and \(\alpha^\vartheta_1(\vartheta)\vee\alpha^\vartheta_2(\vartheta)<1\); the verification for this case is given by Lemma \ref{lem:alphapayoff}. The payoff is in fact a convex combination of \(F^i_\vartheta\) and \(\alpha^\vartheta_j(\vartheta)M^i_\vartheta+(1-\alpha^\vartheta_j(\vartheta))L^i_\vartheta\), except for possibly in the case when \(\alpha^\vartheta_j(\vartheta)=0\) and \(\lambda^\vartheta_M>0\).} 
The maximum can be realized by either the strategy \(\bigl(\indi{t\geq\infty},\indi{t\geq\infty}\bigr)\) or the strategy \(\bigl(\indi{t\geq\vartheta},\indi{t\geq\vartheta}\bigr)\). Consequently, there is a best reply for player \(i\) which is pure and stopping/waiting is strictly optimal iff
\begin{equation}\label{indiff}
L^i_\vartheta-F^i_\vartheta\gtrless\alpha^\vartheta_j(\vartheta)\bigl(L^i_\vartheta-M^i_\vartheta\bigr),
\end{equation}
respectively.\footnote{\label{fn:ext>pure}
An extended mixed strategy can only be strictly superior for some player \(i\) at a jump \(\Delta G^\vartheta_{j}(\tau)\) that is not terminal, in order to secure the payoff \(\Delta G^\vartheta_{j}(\tau)F^i_\tau+(1-G^\vartheta_{j}(\tau))L^i_\tau\) if this is the unique optimal limit of pure (and therefore of any standard mixed) strategies. That limit is not attainable without extensions if \(F^i_\tau>M^i_\tau\).
}
Since \(M^i_\vartheta\leq F^i_\vartheta\), waiting is of course (weakly) optimal whenever \(L^i_\vartheta\leq F^i_\vartheta\) and \(i\) can only be indifferent in that case if \(\alpha^\vartheta_j(\vartheta)\bigl(F^i_\vartheta-M^i_\vartheta\bigr)=0\), with \(\alpha^\vartheta_j(\vartheta)=1\) if \(L^i_\vartheta<F^i_\vartheta\).

If \(L^i_\vartheta>F^i_\vartheta\), \(i\) is indifferent iff
\begin{equation}\label{alphaindiff}
\alpha^\vartheta_j(\vartheta)=\frac{L^i_\vartheta-F^i_\vartheta}{L^i_\vartheta-M^i_\vartheta},
\end{equation}
which is in \((0,1]\) a.s.

Based on \eqref{alphaindiff} we can thus find equilibria of immediate stopping whenever there is a first-mover advantage for both players. These equilibria can be interpreted as preemption in the region \(\{(L^1-F^1)\wedge(L^2-F^2)>0\}\). In order to prepare for subgame-perfect equilibria in asymmetric games where the preemption region is not reached immediately (e.g.\ Theorem \ref{thm:Lsubmart} below), we make sure that if one player is indifferent about becoming leader or follower while the other has a strict preference, then the latter can realize the advantage.\footnote{
The proposition can be modified easily to construct equilibria where no player can realize a first-mover advantage and both \(i=1,2\) receive the payoff \(F^i_\vartheta\), respectively. The corresponding extensions are \(\alpha^\vartheta_i(t)=\indi{t=\tau^\cP(t)}(\indi{L^j_t=M^j_t}+\indi{L^j_t>M^j_t}(L^j_t-F^j_t)/(L^j_t-M^j_t))\).
}

\begin{proposition}\label{prop:eqlL>F}
For any \(\tau\in\T\) let 
\begin{equation*}
\tau^\cP(\tau):=\inf\{u\geq\tau\mid(L^1_u-F^1_u)\wedge(L^2_u-F^2_u)>0\}.
\end{equation*}
Now suppose \(\vartheta\in\T\) satisfies \(\vartheta=\tau^\cP(\vartheta)\) a.s. Then \(\bigl(G^\vartheta_1,\alpha^\vartheta_1\bigr)\), \(\bigl(G^\vartheta_2,\alpha^\vartheta_2\bigr)\) given by
\begin{equation*}
\alpha^\vartheta_i(t)=\begin{cases}
1 & \text{if}\quad\!\begin{aligned}[t]
&t=\tau^\cP(t),\,L^j_t=F^j_t\text{ and}\\
&\!\bigl(L^i_t>F^i_t\text{ or }F^j_t=M^j_t\bigr)
\end{aligned}\\
\\
\indi{L^1_t>F^1_t}\indi{L^2_t>F^2_t}\displaystyle\frac{L^j_t-F^j_t}{L^j_t-M^j_t} & \text{else}
\end{cases}
\end{equation*}
for any \(t\in[\vartheta,\infty)\) and \(G^\vartheta_i=\indi{t\geq\vartheta}\), \(i=1,2\), \(j\in\{1,2\}\setminus i\), are an equilibrium in the subgame at \(\vartheta\).

The resulting payoffs are \(V_i^\vartheta(G_i^\vartheta,\alpha^\vartheta_i,G_j^\vartheta,\alpha^\vartheta_j)=F^i_\vartheta\indi{L^j_\vartheta>F^j_\vartheta}+L^i_\vartheta\indi{L^j_\vartheta=F^j_\vartheta}\).
\end{proposition}

\noindent
{\it Proof:} See appendix.

\begin{remark}
Note that in particular \(\alpha^\vartheta_i(t)=0\) for all \(t<\tau^\cP(t)\). On the ``boundary'' of the preemption region, i.e., if \(t=\tau^\cP(t)\) but \((L^1_t-F^1_t)\wedge(L^2_t-F^2_t)=0\), either \(\alpha^\vartheta_\cdot\) might not be right-continuous for three reasons. First, if we have \(L^j_t=F^j_t=M^j_t\), there might not be a right-hand limit \(\alpha^\vartheta_i(t+)\), which we can accomodate by setting \(\alpha^\vartheta_i(t)=1\) as \(j\) will be completely indifferent. Second, in asymmetric models we have to ensure that a player with a strict first-mover advantage \(L^i_t-F^i_t>0\) can realize it by playing \(\alpha^\vartheta_i(t)=1\) and the other playing \(\alpha^\vartheta_j(t)=0\), cf.\ Theorem \ref{thm:Lsubmart} below. Third, if \(L^j_t>F^j_t\) but \(L^i_t=F^i_t\), then \(\indi{L^i_t>F^i_t}\) might not have a right-hand limit; in this case we have \(\alpha^\vartheta_i(t)=0\) (and \(\alpha^\vartheta_j(t)=1\)).
\end{remark}

In the symmetric case, when the payoff processes do not depend on the players, each player becomes leader or follower with probability \(\frac{1}{2}\) if \(L_\vartheta=F_\vartheta>M_\vartheta\), because then the \(\liminf\) and \(\limsup\) in Definition \ref{def:outcome} are both \(\frac{1}{2}\). This is the same outcome as the result of the Taylor expansion in \cite{FudenbergTirole85} for their smooth, deterministic model. If \(L_\vartheta>F_\vartheta\), there is a positive probability of simultaneous stopping, however, which is the price of preemption.

\(\alpha^\vartheta_i\) in Proposition \ref{prop:eqlL>F} does not depend on \(\vartheta\) (except for the feasibility condition \(\alpha^\vartheta_i\)=0 on \([0,\vartheta)\), of course), so applying the construction to any \(\vartheta\in\T\) induces a subgame-perfect equilibrium if for both \(i=1,2\), \(L^i>F^i\) almost everywhere. To the contrary, if there is not a persistent first-mover advantage, then there can exist many different types of equilibria. One quite general class for which we can use Proposition \ref{prop:eqlL>F} is when the leader's payoff tends to increase in expectation, i.e., when \(L^i\) is a submartingale for each player \(i=1,2\). Then no player wants to stop where \(F^i>L^i\), so stopping results only from preemption.\footnote{
See Theorem \ref{thm:Lsubmart}. For the limits of this logic, however, see Section \ref{subsec:jump}.
}

\subsection{General issues with preemption equilibria}\label{subsec:issupreem}

Such ``purely preemptive'' equilibria are important in the strategic real options literature; a simple deterministic example is shown in Figure \ref{fig:asympreemp}. A number of papers using in fact stochastic models argue that in equilibrium player 1 becomes leader at \(\tau^\cP\), where player 2 becomes follower.\footnote{
See, e.g., \cite{Weeds02}, \cite{PawlinaKort06}, \cite{MasonWeeds10}.
}

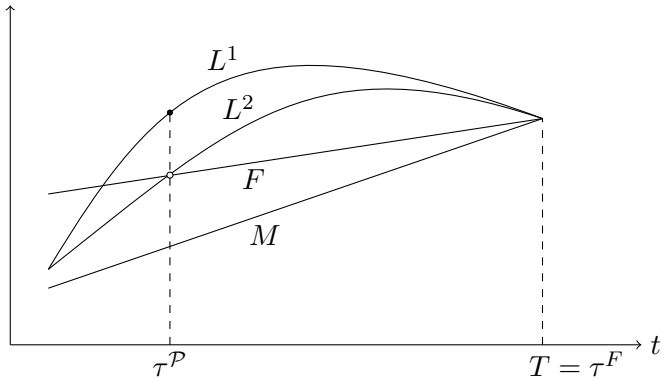
\begin{figure}[ht]
  \centering
  \begin{tikzpicture}[inner sep=0pt,minimum size=0pt,label distance=3pt]
    \draw[->] (0,-1) -- (8.3,-1) node[label=right:$t$] {}; 
    \draw[->] (0,-1) -- (0,3.5) node[] {}; 
    
    \draw[-] (.5,0) .. controls (2,2.5) and (3,3.5) .. (7,2) [] {};
    \draw[-] (.5,0) .. controls (3,2) and (4,3) .. (7,2) [] {};
    \draw[-] (.5,1) -- (7,2) [] {};
    \draw[-] (.5,-.25) -- (7,2) []{};
    
    \node at (2.8,2.8) [] {$L^1$};
    \node at (3,2.15) [] {$L^2$};
    \node at (3.2,1.2) [] {$F$};
    \node at (3.35,.45) [] {$M$};

    \filldraw[fill=white,draw=black] (2.1,1.25) circle (.04);
    \fill[black] (2.1,2.08) circle (.04);
    
    \draw[dashed] (2.1,-1) -- (2.1,2.1) [] {};
    \node at (2.1,-1.25) {$\tau^\cP$};       
    \draw[dashed] (7,-1) -- (7,2) [] {};
    \node at (7.45,-1.25) {$T=\tau^F$}; 
        
  \end{tikzpicture}
  \caption{Preemption with asymmetric leader payoffs}
  \label{fig:asympreemp}
\end{figure}

Stopping must occur no later than at \(\tau^\cP\) in equilibrium, because the players would try to preempt each other where both have a strict first mover-advantage \(L^i>F^i\). In this deterministic example it also seems clear that no player wants to stop at any \(t<\tau^\cP\), because the payoffs keep increasing. There are two general issues in supporting ``stopping at \(\tau^\cP\)'' as an equilibrium. 

First, for stopping really to occur at \(\tau^\cP\), player 1 must not be able to realize a further increase in \(L^1\). Exploiting the increasing payoff can only be prevented by a (credible) threat of player 2 to stop sufficiently quickly after \(\tau^\cP\) if player 1 does not stop. There is no such threat in the mentioned papers. There, the follower still has an option to stop (when to enter a market given that the rival has already entered), \(F\) is the corresponding value function, and the solution is to stop at some later time \(\tau^F\) (when the market has sufficiently grown to sustain duopoly), say here at the terminal time \(T\). That player 2 stops at \(\tau^F\) cannot be an equilibrium strategy. It can only be the \emph{outcome} of both players trying to stop at or immediately after \(\tau^\cP\) if no one has stopped before, player 1 succeeding and player 2 reacting to the changed history by waiting until \(\tau^F\).

We model a proper game theoretic equilibrium of preemption on \((\tau^\cP,T)\) with the strategies given in Proposition \ref{prop:eqlL>F}, such that the game ends immediately at any point in this region, even if some player deviates unilaterally. Both players are also willing to stop immediately because the extensions \(\alpha^\vartheta_i\) allow to control the probability of simultaneous stopping. The identity of who stops first (1, 2 or both) determines the entire outcome, the probabilities of which depend on the relation between \(L^i\), \(F^i\) and \(M^i\), \(i=1,2\) (through the equilibrium strategies \(\alpha^\vartheta_i\)). Any eventual follower's stopping decision is encoded in \(F^i\).

Second, on \([0,\tau^\cP)\) the players must be willing to wait until \(\tau^\cP\). In the deterministic example, the players do not want to stop at any \(t<\tau^\cP\) because the payoffs keep increasing. For waiting \emph{until} \(\tau^\cP\) to be an equilibrium, however, (i) player 1 must also be sure to become leader and (ii) player 2 must be sure that there is no possibility of simultaneous stopping. We support exactly that outcome in Proposition \ref{prop:eqlL>F}.\footnote{
For this outcome we cannot require right-continuity of the \(\alpha^{t}_i(\cdot)\) at \(\tau^\cP\). Note that for any \(t\leq\tau^\cP\), \(\alpha^{t}_1(\tau^\cP+)=\lim_{u\searrow\tau^\cP}\indi{L^1_u>F^1_u}\indi{L^2_u>F^2_u}\frac{L^2_u-F_u}{L^2_u-M_u}=0\) and \(\alpha^{t}_2(\tau^\cP+)=\lim_{u\searrow\tau^\cP}\indi{L^1_u>F^1_u}\indi{L^2_u>F^2_u}\frac{L^1_u-F_u}{L^1_u-M_u}>0\). Defining both \(\alpha^{t}_i(\tau^\cP)\) by right-continuity would imply player 2 to become leader and player 1 to become follower for sure.
}
Otherwise each player would want to stop ever more closely before \(\tau^\cP\), which cannot be an equilibrium. 

In the typical stochastic models, the state of a Markovian process takes the role of time in the deterministic example. Then, however, it is not at all obvious whether the players really want to wait until the preemption region is reached, even if one has a similar ordering of the payoff ``functions'' \(L^i\) and \(F^i\) on the state space.\footnote{
In particular the argument that \(F^i>L^i\) is pointless if the opponent does not intend to stop.
}
The state does not proceed as linearly as time, so one needs to study the related stopping problems.

\subsection{Subgame-perfect preemption equilibria}\label{subsec:SGPpreem}

A reasonable stochastic analogue of a deterministic increasing leader payoff is assuming \(L^i\) to increase in expectation, i.e., to be a submartingale at least outside the preemption region. Then we get a very general existence result, without any particular assumption on the underlying stochastics. Continuity is however an important property; cf.\ the example in Section \ref{subsec:jump}. An alternative argument to obtain such ``purely preemptive'' equilibria is presented in Section \ref{subsec:sympreemp}, based on \(F^i\) being indeed the value process of some follower's stopping problem.

\begin{theorem}\label{thm:Lsubmart}
Assume that each \(L^i\) is a submartingale and that \(L^i\) and \(F^i\) are a.s.\ continuous, \(i=1,2\). Then there exists a subgame-perfect equilibrium \(\bigl(G_1,\alpha_1\bigr)\), \(\bigl(G_2,\alpha_2\bigr)\) with \(\alpha^\vartheta_i\) given by Proposition \ref{prop:eqlL>F} and \(G^\vartheta_i=\indi{t\geq\tau^\cP(\vartheta)}\) for all \(\vartheta\in\T\) and \(i=1,2\).

The resulting payoffs are 
\begin{equation*}
V_i^\vartheta(G_i^\vartheta,\alpha^\vartheta_i,G_j^\vartheta,\alpha^\vartheta_j)=\begin{cases}
E\bigl[L^i_{\tau^\cP(\vartheta)}\bigm|\F_\vartheta\bigr] &\text{if }\{\vartheta<\tau^\cP(\vartheta)\}\text{ or }L^j_\vartheta=F^j_\vartheta,\\
F^i_\vartheta &\text{else}.
\end{cases}
\end{equation*}
\end{theorem}

\noindent
{\it Proof:} See appendix.

\begin{remark}
It would suffice that each \(L^i\) is a semimartingale, hence of the form \(L^i=M^i+A^i\) with a martingale \(M^i\) and a finite-variation process \(A^i\) (e.g., \(L^i\) is a diffusion) and that each \(A^i\), which inherits continuity from \(L^i\), is non-decreasing outside the preemption region \(\{(L^1-F^1)\wedge(L^2-F^2)>0\}\).
\end{remark}

\section{Illustrative examples}\label{sec:exm}

In the following we present a few examples to illustrate the equilibrium concept developed in this paper. Section \ref{subsec:grab} shows the workings of extended mixed strategies in a simple asymmetric game. Section \ref{subsec:sympreemp} presents the standard model from the theory of investment under uncertainty and derives mixed strategies to support equilibrium outcomes proposed for similar models in the literature as subgame-perfect equilibria. Finally, Section \ref{subsec:jump} gives an example showing that the logic used in Theorem \ref{thm:Lsubmart} is sensitive to jumps.

\subsection{Example: the cross-country grab-the-dollar-game}\label{subsec:grab}

We illustrate our definition of equilibrium with a stochastic version of the grab-the-dollar game. It serves to illustrate why it is important to allow for adapted (instead of $\F_\tau$-measurable) strategies in the subgames starting at some stopping time $\tau$. With stochastic payoffs, it is generally impossible to fix one's own strategy independently of the development of the state variables.

In our game, we consider an American and her European friend playing the grab-the-dollar game. When the European wins the dollar, he has to turn it into Euros, at a stochastic exchange rate given by an adapted, right-continuous process $X_t>0$. If both players grab the dollar at the same time, they pay a penalty of $1$ currency unit in their local currency.

Let $0<\exp(-r)<1$ be the discount factor for both players.
The payoffs are thus $L^1_t=\exp(-rt)$ for the American and $L^2_t=X_t \exp(-rt)$ for the European if they win the dollar, respectively. As in the usual grab-the-dollar game, $F^i=0$, and we set $M^1_t=M^2_t=-\exp(-rt)$  for the simultaneous entry payoffs.

\begin{proposition}
  A subgame-perfect equilibrium is given by the strategies $G^\vartheta_i(t) = 1_{\{t \ge \vartheta\}}, i=1,2$ (``grab immediately'') and
\begin{flalign*}  
&& \alpha^\vartheta_1(t) &= \frac{X_t}{1+X_t}\,  1_{\{t\ge\vartheta\}} && \\
&\text{and} \\
&& \alpha^\vartheta_2(t) &= \frac{1}{2}\, 1_{\{t\ge\vartheta\}} && 
\end{flalign*}
for all stopping times $\vartheta\in\T$ and $t\ge 0$.
\end{proposition}

This equilibrium is an application of Proposition \ref{prop:eqlL>F}. Feasibility and time-consistency of the strategies are straightforward to verify.

\subsection{Example: preemptive market entry}\label{subsec:sympreemp}

Now we present a typical example of a strategic real option exercise problem and show how to derive subgame-perfect equilibrium strategies for similar models. 

Two firms \(i=1,2\) have the opportunity to invest irreversibly in the same market. The return from investment, in particular the profit flow from operating in the market, is uncertain. Assume the duopoly profit flow when both firms have invested is given by the stochastic process $X=(X_t)_{t\geq 0}$. If only one firm is present in the market it can realize a monopoly markup and increase the profit flow to \(MX_t\) for some constant \(M>1\). To obtain possibly explicit results and because this is the most familiar model in the literature we let $X$ be a geometric Brownian motion; the idea of proof will apply to more general processes, though. Hence, we assume \(X\) is the unique strong solution to the stochastic differential equation
\begin{equation*}
\frac{dX}{X}=\mu\,dt+\sigma\,dB
\end{equation*}
with given initial value \(X_0=x\in\R_+\), where \(B\) is a Brownian motion and \(\mu\), \(\sigma\) some constants. There is a sunk cost of investment \(I>0\). We assume that profits are discounted at a common and constant rate \(r>\max(\mu,0)\), which ensures integrability of our payoff processes and finiteness of the subsequent stopping problems.\footnote{
With \(r>\max(\mu,0)\), \((e^{-rt}X_t)\) is bounded by an integrable random variable. Indeed, for \(\sigma>0\) we have \(\sup_te^{-rt}X_t=X_0e^{\sigma Z}\) with \(Z=\sup_tB_t-t(r-\mu+\sigma^2/2)/\sigma\), which is exponentially distributed with rate \(2(r-\mu)/\sigma+\sigma\) (see, e.g., \cite{RevuzYor}, Exercise (3.12) $4^\circ$). Thus, \(E[\sup_te^{-rt}X_t]=X_0(1+\sigma^2/2(r-\mu))\in\R_+\), implying that \((e^{-rt}X_t)\) is of class (D); analogously for \(\sigma<0\).
}

If firm \(i\) invests at time \(t\) as the leader, the other firm that becomes follower still has the option to invest at some later time. The follower's payoff process \(F\) is thus the value function of the optimal stopping problem
\begin{flalign*}
&& F_t:={}&\sup_{\tau\geq t}E\biggl[\int_\tau^\infty e^{-rs}(X_s-rI)\,ds\biggm|\F_t\biggr] &&\nonumber\\
\nonumber\\
&& ={}&\begin{cases}
\displaystyle e^{-rt}\left(\frac{X_t}{x^F}\right)^{\beta_1}\biggl(\frac{x^F}{r-\mu}-I\biggr)\quad\text{ if }X_t<x^F\\
\\
\displaystyle e^{-rt}\biggl(\frac{X_t}{r-\mu}-I\biggr)\quad\text{ else}.
\end{cases} &&
\end{flalign*}
The explicit solution is a standard result where the investment threshold for the state process \(X\) is given by
\begin{equation*}
x^F=\frac{\beta_1}{\beta_1-1}(r-\mu)I>rI,
\end{equation*}
and \(\beta_1>1\) is the positive root of the quadratic equation \(0=\frac{1}{2}\sigma^2\beta(\beta-1)+\mu\beta-r\).

If the leader invests at any stopping time \(\vartheta\), we denote the optimal investment time of the follower by
\begin{equation*}
\tau^F(\vartheta)=:\inf\{s\geq\vartheta\mid X_s\geq x^F\},
\end{equation*}
whence the leader enjoys the monopoly profit only on \([\vartheta,\tau^F(\vartheta))\). Thus the leader's payoff process \(L\) is given by
\begin{flalign*}
&& L_t:={}&E\biggl[\int_t^{\tau^F(t)}e^{-rs}(MX_s-rI)\,ds+\int_{\tau^F(t)}^\infty e^{-rs}(X_s-rI)\,ds\biggm|\F_t\biggr] &&\nonumber\\
\nonumber\displaybreak[0]\\
&& ={}&\begin{cases}
\displaystyle e^{-rt}\left(\frac{MX_t}{r-\mu}-I+\left(\frac{X_t}{x^F}\right)^{\beta_1}\biggl(\frac{x^F(1-M)}{r-\mu}\biggr)\right)\quad\text{ if }X_t<x^F\\
\\
\displaystyle F_t\quad\text{ else}.
\end{cases} &&
\end{flalign*}
Finally, the payoff from simultaneous investment is simply given by
\begin{flalign*}
&& M_t:={}&E\biggl[\int_t^\infty e^{-rs}(X_s-rI)\,ds\biggm|\F_t\biggr] &&\\
\\
&& ={}&e^{-rt}\left(\frac{X_t}{r-\mu}-I\right), &&
\end{flalign*}
i.e., \(M_t=F_t=L_t\) whenever \(X_t\geq x^F\). 

In order to determine when there is a first- or second-mover advantage, we can rely on the strong Markov property and identify the corresponding regions of the state space of the process \(X\): there exists a unique \(x^P\in(0,x^F)\) such that\footnote{\label{fn:xP}
Consider time 0, recalling that \(X_0=x\in\R_+\). We express the dependence on the starting value by writing \(L_0=L(x)\) and \(F_0=F(x)\). Then \(L(x)-F(x)\to-I<0\) for \(x\to 0\). Further,
\begin{equation*}
\frac{\partial(L(x)-F(x))}{\partial x}=\frac{M}{r-\mu}+\frac{\beta_1}{x^F}\left(\frac{x}{x^F}\right)^{\beta_1-1}\biggl(I-\frac{Mx^F}{r-\mu}\biggr).
\end{equation*}
The term in the last parentheses is strictly negative by \(M>1\). Thus the displayed derivative is strictly decreasing in \(x\), starting at \(M/(r-\mu)>0\) for \(x=0\) and ending at \((1-M)(\beta_1-1)/(r-\mu)<0\) for \(x=x^F\), where \(L(x^F)=F(x^F)\). Hence there exists a unique \(x^P\in(0,x^F)\) such that \(L(x)-F(x)<0\) iff \(x\in[0,x^P)\) and \(L(x)-F(x)>0\) iff \(x\in(x^P,x^F)\).
} 
\begin{flalign*}
\begin{cases}
L_t<F_t\quad\text{ iff }X_t\in[0,x^P),\\
\\
L_t>F_t\quad\text{ iff }X_t\in(x^P,x^F).
\end{cases}
\end{flalign*}
Consequently, the interval \(\cP=(x^P,x^F)\) is the preemption region for the driving process \(X\). On \(\{X\in\cP\}\) we have equilibria of immediate stopping with coordination by extended mixed strategies following Proposition \ref{prop:eqlL>F}, resulting in the immediate payoffs \(F\). Let 
\[
\tau^\cP(\vartheta):=\inf\{s\geq\vartheta\mid X_s\in\cP\}
\] 
denote the hitting time of the preemption region after any stopping time \(\vartheta\in\T\). 

On \(\{X_\vartheta<x^P\}\) we have \(L_\vartheta<F_\vartheta=E[F_{\tau^\cP(\vartheta)}|\F_\vartheta]=E[L_{\tau^\cP(\vartheta)}|\F_\vartheta]\), since \(F\) is a martingale in its continuation region up to \(\tau^F(\vartheta)>\tau^\cP(\vartheta)\). From this observation one can prove that stopping is dominated on \([\vartheta,\tau^\cP(\vartheta))\) where \(X_\vartheta<x^P\). In particular this is obvious if the opponent does not stop on that interval, because then each player's stopping problem has at least the value \(E[F_{\tau^\cP(\vartheta)}|\F_\vartheta]=E[L_{\tau^\cP(\vartheta)}|\F_\vartheta]\).\footnote{
This value obtains for player \(i\) from stopping at \(\tau^\cP(\vartheta)\) if \(\Delta G^\vartheta_j(\tau^\cP(\vartheta))(F_{\tau^\cP(\vartheta)}-M_{\tau^\cP(\vartheta)})=0\), from stopping after \(\tau^\cP(\vartheta)\) if \(\Delta G^\vartheta_j(\tau^\cP(\vartheta))=1\), or otherwise from the limit of stopping at \(\tau^\cP(\vartheta)+\varepsilon\) where \(\varepsilon\searrow 0\). In the last case the probability of simultaneous stopping converges to 0 since any \(G^\vartheta_j\) of the opponent needs to be right-continuous, and (right-)continuity of \(L\) induces the claimed limit value (which will however not be \emph{attained} by stopping at \(\tau^\cP(\vartheta)\), due to \(\Delta G^\vartheta_j(\tau^\cP(\vartheta))(F_{\tau^\cP(\vartheta)}-M_{\tau^\cP(\vartheta)})>0\)).
}

Finally, let \(\cM=[x^F,\infty)\), so that \(M_t=F_t\) iff \(X_t\in\cM\). On \(\{X\in\cM\}\) we have equilibria of simultaneous stopping. In fact, given that the preemption payoffs on \(\{X\in\cP\}\) are \(F\), stopping is even strictly dominant on \(\{X\in\cM\}\), because then the drift of the supermartingale F is the strictly negative forgone revenue \(-e^{-rt}(X_t-rI)\,dt\). We can summarize as follows, letting \(\tau^\cM(\vartheta):=\inf\{s\geq\vartheta\mid X_s\in\cM\}\) denote the hitting time of \(\cM\).

\begin{proposition}\label{prop:eqlpreemp}
There exists a subgame-perfect equilibrium for the present preemption example with extended mixed strategies \(\bigl(G_1,\alpha_1\bigr)\) and \(\bigl(G_2,\alpha_2\bigr)\) given as follows.  For any \(i=1,2\) and \(\vartheta\in\T\) set
\begin{flalign*}
&& \alpha^\vartheta_i(t)&=\begin{cases}\displaystyle
\frac{L_t-F_t}{L_t-M_t} & \text{if}\quad L_t>F_t\Leftrightarrow X_t\in\cP=(x^P,x^F)\\
\\
1 & \text{if}\quad F_t=M_t\Leftrightarrow X_t\in\cM=[x^F,\infty)\\
\\
0 & \text{else}
\end{cases} &&\\
\text{and}\\
&& G^\vartheta_i(t)&=\indi{t\geq\tau^\cP(\vartheta)\wedge\tau^\cM(\vartheta)} &&
\end{flalign*}
for any \(t\in[\vartheta,\infty]\).

The resulting payoffs are \(V_i^\vartheta(G_i^\vartheta,\alpha^\vartheta_i,G_j^\vartheta,\alpha^\vartheta_j)=E\bigl[F_{\tau^\cP(\vartheta)\wedge\tau^\cM(\vartheta)}\bigm|\F_\vartheta\bigr]\).
\end{proposition}

\noindent
{\it Proof:} See appendix.
\medskip

Fully explicitly, we have 
\begin{equation*}
\alpha^\vartheta_i(t)=\frac{\displaystyle\frac{MX_t}{r-\mu}-I+\left(\frac{X_t}{x^F}\right)^{\beta_1}\biggl(I-\frac{Mx^F}{r-\mu}\biggr)}{\displaystyle\frac{(M-1)X_t}{r-\mu}+\left(\frac{X_t}{x^F}\right)^{\beta_1}\biggl(\frac{(1-M)x^F}{r-\mu}\biggr)},\qquad X_t\in\cP.
\end{equation*}
The extensions \(\alpha^\vartheta_i(\cdot)\) here inherit continuity from the payoff processes.\footnote{
This is clear except for possibly two cases. If \(X_t=x^P\), then \(L_t=F_t\) and \(\lim_{u\to t}\indi{L_u>F_u}(L_u-F_u)/(L_u-M_u)=0\) as \(L_t>M_t\) for \(X_t\not\in\cM\). If \(X_t=x^F=\partial\cM\), then 
\begin{equation*}\lim_{u\to t}\alpha^\vartheta_i(u)=\lim_{u\to t}\bigl[\indi{L_u>F_u}(L_u-F_u)/(L_u-M_u)+\indi{F_u=M_u}\bigr]=1=\alpha^\vartheta_i(t),\end{equation*} 
because \(\lim_{x\nearrow x^F}(L(x)-F(x))/(L(x)-M(x))=1\) by l'H\^opital (with the notation of footnote \ref{fn:xP}, writing also \(M(x)=M_0\) for \(X_0=x\)).
}
As remarked more generally in the context of Proposition \ref{prop:eqlL>F}, each firm eventually becomes leader or follower with probability \(\frac12\) in any subgame that starts with \(X_\vartheta\leq x^P\).

\subsection{Sensitivity to jumps}\label{subsec:jump}

The following simple economic example illustrates the sensitivity of the logic brought forward in Sections \ref{subsec:issupreem}, \ref{subsec:SGPpreem}~-- that stopping is dominated where \(L^i<F^i\) and \(L^i\) is (strictly) increasing in expectation~-- to continuity of the payoff processes. We will show that in \emph{any} equilibrium of the example, stopping occurs strictly before reaching the preemption region (where both players have a first-mover advantage), although the following regularity properties hold: For each \(i=1,2\) the leader payoff process \(L^i\) is a strict submartingale that is further upper-semicontinuous (the usual regularity condition for optimal stopping), and the follower process \(F^i\) is continuous; the preemption region is non-empty with probability 1. We will actually construct subgame-perfect equilibria with immediate stopping at any time.

Consider two rival firms that can each make an investment to diversify to a new product. If only one firm invests, it will take up the new market, while the other then becomes monopolist for the old product. The latter is worth a net present value \(c\) at the time of investment. Initially only firm 1 has developed a profitable technology for switching to the new product, such that it can invest into it any time. The technology keeps improving, however, so the value of capturing the new market is increasing in time, even if discounted to time 0. Firm 2 initially has an inferior technology, such that investing into the new product would only yield it a net present value 1 at the time of investment. However, firm 2 can catch up to the superior technology at the hazard rate \(\lambda>0\), after which it could realize the same profit as firm 1. As usual, simultaneous investment is the worst outcome.

We model the payoff processes as follows:
\begin{alignat*}{2}
L^1&=\bigl(a-e^{-rt}b\bigr)_{t\in[0,\infty)}, &  \quad&a=2+\tfrac{r}{\lambda},\ b=\tfrac{r}{r+\lambda},\\
L^2&=\bigl(e^{-rt}\indi{t<T}+(a-e^{-rt}b\bigr)\indi{t\geq T})_{t\in[0,\infty)}, && \\
F^1&=F^2=\bigl(e^{-rt}c\bigr)_{t\in[0,\infty]}, & &c\in[1,1+b),\\
&\geq M^1,\,M^2.
\end{alignat*}
The choice of the \(M^i\), the payoffs for simultaneous investment, is unimportant as long as there is a (weak) penalty.\footnote{
Also \(F^1\) can be modified, for the argument we only need \(F^1\leq F^2\).
}
\(r>0\) is the fixed discount rate. \(T\) is the random time at which firm 2 catches up, it is exponentially distributed with parameter \(\lambda\) and defined on some stochastic basis \((\Omega,\F,P)\). Assume this is the only uncertainty, i.e., the filtration \(\mathbf{F}=(\F_t)_{t\in[0,\infty)}\) is generated by the process \((\indi{t\geq T})\). 

Notice for the sake of the argument that \(F^2_t>L^2_t\) on \([0,T)\) for \(c>1\). All processes except for \(L^2\) (and \(M^1\), \(M^2\), if we like) are continuous. \(L^2\) is continuous from the right and upper-semicontinuous from the left, because \(a>2>(1+b)\geq e^{-rt}(1+b)\). \(L^1\) is strictly increasing and also \(L^2\) is a strict submartingale.\footnote{\label{fn:Lsubmart}
Fix two times \(0\leq s<t\). On \(\{s\geq T\}\) we have \(E[L^2_t\mid\F_s]-L^2_s=e^{-rs}b-e^{-rt}b>0\). On \(\{s<T\}\),
\begin{align*}
E[L^2_t\mid\F_s]&=e^{-rt}P[T>t\mid T>s]+(a-e^{-rt}b)P[T\leq t\mid T>s]\\
&=e^{-rt}e^{-\lambda(t-s)}+(a-e^{-rt}b)(1-e^{-\lambda(t-s)})
\end{align*}
and thus
\begin{gather*}
\partial E[L^2_t\mid\F_s]/\partial t=-(r+\lambda)e^{-rt}e^{-\lambda(t-s)}(1+b)+a\lambda e^{-\lambda(t-s)}+re^{-rt}b>0\quad\forall t>s\geq 0\\
\Leftrightarrow\quad a\lambda e^{rt}+re^{\lambda(t-s)}b>(r+\lambda)(1+b)\quad\forall t>s\geq 0\\
\Leftrightarrow\quad a\lambda+rb\geq(r+\lambda)(1+b)\\
\Leftrightarrow\quad a\geq(1+b)+\tfrac{r}{\lambda}.
\end{gather*}
Hence, \(E[L^2_t\mid\F_s]>\lim_{u\searrow s}E[L^2_u\mid\F_s]=e^{-rs}=L^2_s\). 
}
The key property for the intended result is that \(L^2\) is strictly decreasing up to \(T\) and exceeds the expected value of becoming follower at \(T\) as we will show, so firm 2 will stop immediately if there is too little chance to realize \(L^2\) after \(T\). 

The preemption region is \(\{(L^1-F^1)\wedge(L^2-F^2)>0\}=\{T\leq t<\infty\}\subset\Omega\times[0,\infty]\) since \(a>b+c\geq e^{-rt}b+e^{-rt}c\). Its hitting time, starting from time 0, is 
\begin{equation*}
\tau^\cP:=\inf\{s\geq 0\mid(L^1_s-F^1_s)\wedge(L^2_s-F^2_s)>0\}=T.
\end{equation*}
Now suppose there is an equilibrium with preemption using extended mixed strategies inside the preemption region, i.e., with continuation payoff \(F^i_\vartheta\) for each firm \(i=1,2\) and any stopping time \(\vartheta>\tau^\cP=T\). Assume that at \(T\), however, the firms can agree on events \(A_1\) and \(A_2\) with \(P[A_1\cap A_2]=0\), such that firm \(i\) even obtains \(L^i_T\) on \(A_i\) by playing \(\alpha_i(T)=1\) and playing \(\alpha_i(T)=0\) on \(A_j\), \(i,j\in\{1,2\}\), \(i\not=j\). Let \(A_i\) in fact be exactly the event where \(\alpha_i(T)=1\), so \(i\)'s payoff on \(A_i^c\) is \(F^i_T\). If there is no stopping on \([0,T)\), then the expected payoffs are 
\begin{flalign*}
&& E[L^1_T\indi{A_1}+F^1_T\indi{A_1^c}]&=E[(a-e^{-rT}b)\indi{A_1}+e^{-rT}c\indi{A_1^c}] &&\\
\text{and}\\
&& E[L^2_T\indi{A_2}+F^2_T\indi{A_2^c}]&\leq E[(a-e^{-rT}b)\indi{A_1^c}+e^{-rT}c\indi{A_1}]. &&
\end{flalign*}
The estimate follows from \(L^2_T>F^2_T\) and \(P[A_2\setminus A_1^c]=P[A_1\setminus A_2^c]=0\). The sum of the expected payoffs is not more than \(a-bE[e^{-rT}]+cE[e^{-rT}]=a-(b-c)\lambda/(r+\lambda)<a-b+1=L^1_0+L^2_0\), contradicting the hypothesized equilibrium.

Now suppose there is an equilibrium with mixed strategies \(G^T_i\) in the preemption region. From \(T\) onwards, the payoff processes are deterministic and continuous. \cite{HendricksWilson92} show that mixed equilibrium strategies must satisfy
\begin{equation*}
\frac{dG^T_i}{1-G^T_i}=\frac{dL^j}{L^j-F^j}=\frac{re^{-rt}b\,dt}{a-e^{-rt}(b+c)},\qquad i,j\in\{1,2\},\,i\not=j,
\end{equation*}
where continuous mixing occurs. We here have \(0<(a-e^{-rt}(b+c))^{-1}\leq(a-b-c)^{-1}<\lambda(r+\lambda)/(r^2+\lambda^2)<\infty\) and thus \(\int_T^\infty(1-G^T_i)^{-1}\,dG^T_i<\infty\), \(i=1,2\), so there would remain some mass for stopping at \(t=\infty\) for both firms. However, \(M^i_\infty\leq F^i_\infty=0<\lim_{t\nearrow\infty}L^i_t=a\) for both \(i=1,2\), so there cannot be joint stopping at \(t=\infty\). 

It follows from Theorems 2 and 3 in \cite{HendricksWilson92} that at \(T\) there only exist equilibria with at least one firm \(i\) stopping immediately, i.e., \(G^T_1(T)\vee G^T_2(T)=1\). A firm that stops receives not more than \(L^i_T\) and the respective other receives \(F^j_T\) (there may be simultaneous stopping if \(F^i_T=M^i_T\), \(i=1,2\)). Now the previous argument based on the sets \(A_i=\{G^T_i(T)=1\}\) applies again, so stopping must occur strictly before \(T\) in equilibrium.

The estimate above extends to any time \(t\in[0,T)\), taking conditional expectations. In fact, there exist the following subgame-perfect equilibria with either firm stopping immediately on \([0,T)\). Set
\begin{equation*}
\alpha^\vartheta_1(t)=\indi{t\geq T}\frac{L^2_t-F^2_t}{L^2_t-M^2_t}\quad\text{and}\quad\alpha^\vartheta_2(t)=\indi{t\geq T}\frac{L^1_t-F^1_t}{L^1_t-M^1_t}
\end{equation*}
for all stopping times \(\vartheta\) and \(t\in[0,\infty)\). Now pick \(i,j\in\{1,2\}\), \(i\not=j\) and set \(G^\vartheta_i=\indi{t\geq\vartheta}\) (stop immediately) and \(G^\vartheta_j=\indi{t\geq\vartheta\vee T}\) (stop at preemption region) for all \(\vartheta\).  On \(\{\vartheta\geq T\}\) there is preemption with payoffs \(F^1_\vartheta\), \(F^2_\vartheta\). On \(\{\vartheta<T\}\), \(j\) cannot deviate profitably because \(M^j\leq F^j\). Firm \(i\) could wait until any \(\tau\geq\vartheta\) to obtain \(L^i_\tau\indi{\tau<T}+F^i_T\indi{\tau\geq T}\). The process \((L^i_t\indi{t<T}+F^i_T\indi{t\geq T})_{t\geq 0}\) is however a strict supermartingale for both \(i=1,2\) with our parametrization.\footnote{
Fix two times \(0\leq s<t\). On \(\{s<T\}\) we have \(E[L^1_t\indi{t<T}+F^1_T\indi{t\geq T}\mid\F_s]=(a-e^{-rt}b)e^{-\lambda(t-s)}\) (for \(F^1\equiv 0\)) and
\begin{equation*}
E[L^2_t\indi{t<T}+F^2_T\indi{t\geq T}\mid\F_s]=e^{-rt}e^{-\lambda(t-s)}+e^{-rs}c\tfrac{\lambda}{r+\lambda}(1-e^{-(r+\lambda)(t-s)})
\end{equation*}
(resp.\ \(E[L^1_t\indi{t<T}+F^1_T\indi{t\geq T}\mid\F_s]\leq (a-e^{-rt}b)e^{-\lambda(t-s)}+e^{-rs}c\tfrac{\lambda}{r+\lambda}(1-e^{-(r+\lambda)(t-s)})\) for any other choice of \(F^1\leq F^2\)). Now one can continue as in footnote \ref{fn:Lsubmart}.
}
Thus, stopping immediately is indeed optimal.

\section{Conclusion}\label{sec:conc}

We have presented a framework for subgame-perfect equilibria of stochastic timing games, where the notion of a subgame is linked to that of a stopping time. The latter are the feasible (random) time quantities to plan an action, and hence also to revise any plan. In continuous time the events that are decidable by stopping times are generally more complex than combinations of events from a collection of deterministic times (in contrast to discrete time).

In many timing games mixed strategies play an important role as we have argued, either for equilibrium existence or to resolve any strategic conflicts (about roles with differing amenities) within the game. Concerning both aspects we have taken up the concept of \cite{FudenbergTirole85} to extend mixed strategies of a more classical sense towards improved consistency with limits from discrete time~-- instead of modifying the underlying game payoffs to circumvent coordination problems (as frequently done by coin tossing, e.g.). The extensions allow to solve coordination problems about roles in equilibrium endogenously and there is generally a price to pay for preemption, in terms of coordination failure. 

The generalization to stochastic models has not been straightforward, however. With necessarily weaker regularity restrictions on strategies, one has to be very careful how to define unique outcomes that are possibly consistent with the motivating limits. With this framework fixed, we obtained a general characterization of equilibria in subgames of preemption type. The main issue is then to verify equilibria in subgames where there is not a first-mover advantage for both players, by solving appropriate stopping problems. We have covered some important cases for applications and thus provided a foundation for the literature on strategic real options, which is often lacking rigorous arguments.

Of course the models where stopping occurs only due to preemption are just a special~-- if important~-- class. We presented an example which seems to follow the logic of that class, but where stopping occurs much earlier to avoid preemption later on. There may also be other forms of stopping, with the players using continuous distribution functions. The paper \cite{StegThijssen14} presents such a case, where the regime of a game in which two players have a switching option moves randomly between first- and second-mover advantages. Then both components of the extended mixed strategies introduced here play important roles.

\newpage
\appendix

\section{Appendix}

\subsection{Proofs}

\begin{proof}[{\bf Proof of Proposition \ref{prop:eqlL>F}}]
By construction, \(G^\vartheta_i\) and \(\alpha^\vartheta_i\) are a.s.\ \nbd{[0,1]}valued. \(G^\vartheta_i\) is right-continuous, non-decreasing, attaining 1 where \(\alpha^\vartheta_i(t)>0\), \(t\geq\vartheta\). \(\alpha^\vartheta_i\) takes values in \((0,1)\) only where \((L^1_t-F^1_t)\wedge(L^2_t-F^2_t)>0\), where it is indeed right-continuous.

\(G^\vartheta_i\) is adapted. \(\alpha^\vartheta_i\) is progressively measurable because we can represent it using the process \((\indi{t=\tau^\cP(t)})\), which is the upper-right-continuous modification of the optional process \((\indi{(L^1_t-F^1_t)\wedge(L^2_t-F^2_t)>0})\) and therefore progressively measurable by Theorem IV.33 (c) in \cite{DellacherieMeyer78}~-- our \(\filt{F}\) satisfies the usual conditions.

By \(\vartheta=\tau^\cP(\vartheta)\) we also have \(\vartheta=\hat\tau^\vartheta_i=\inf\{u\geq\vartheta\mid\alpha^\vartheta_i(u)>0\}\) and \(L^i_\vartheta\geq F^i_\vartheta\) a.s., \(i=1,2\). By our observations about \eqref{alphaindiff} we have indifference where \((L^1_\vartheta-F^1_\vartheta)\wedge(L^2_\vartheta-F^2_\vartheta)>0\), implying payoff \(F^1_\vartheta\), resp.\ \(F^2_\vartheta\), and it only remains to verify that each player \(i\) obtains \(\max\{F^i_\vartheta,\alpha^\vartheta_j(\vartheta)M^i_\vartheta+(1-\alpha^\vartheta_j(\vartheta))L^i_\vartheta\}\) in the cases \(L^1_\vartheta=F^1_\vartheta\) or \(L^2_\vartheta=F^2_\vartheta\). Now fix \(i,j\in\{1,2\}\), \(i\not=j\). Consider first \(L^i_\vartheta=F^i_\vartheta\). If \(L^j_\vartheta>F^j_\vartheta\) then \(\alpha^\vartheta_j(\vartheta)=1\) and \(\alpha^\vartheta_i(\vartheta)=0\) is optimal. If also \(L^j_\vartheta=F^j_\vartheta\) then still \(\alpha^\vartheta_j(\vartheta)=1\) where \(F^i_\vartheta=M^i_\vartheta\), where \(i\) is completely indifferent. Finally, where \(L^j_\vartheta=F^j_\vartheta\) and \(F^i_\vartheta>M^i_\vartheta\), \(\alpha^\vartheta_j(\vartheta)=\alpha^\vartheta_j(\vartheta+)=0\). Then \(\alpha^\vartheta_i(\vartheta)=1\) where \(F^j_\vartheta=M^j_\vartheta\) is optimal, as is \(\alpha^\vartheta_i(\vartheta)=\alpha^\vartheta_i(\vartheta+)=0\) where \(F^j_\vartheta>M^j_\vartheta\), since then \(\lambda^\vartheta_M=0\). In all these cases the payoff is \(F^i_\vartheta=L^i_\vartheta\).

Finally consider \(L^i_\vartheta>F^i_\vartheta\) and \(L^j_\vartheta=F^j_\vartheta\), whence \(\alpha^\vartheta_j(\vartheta)=0\) and thus \(\alpha^\vartheta_i(\vartheta)=1\) is optimal. In this case the payoff is \(L^i_\vartheta\).
\end{proof}

\begin{proof}[{\bf Proof of Theorem \ref{thm:Lsubmart}}]
Admissibility of the strategies for any \(\vartheta\in\T\) is obtained as in the proof of Proposition \ref{prop:eqlL>F} and time-consistency is obvious. For any \(\vartheta\in\T\) and \(i=1,2\), \(\bigl(G^\vartheta_i,\alpha^\vartheta_i\bigr)\) is also admissible for the subgame starting at \(\tau^\cP(\vartheta)\) and Proposition \ref{prop:eqlL>F} shows that \(\bigl(G^\vartheta_1,\alpha^\vartheta_1\bigr)\) and \(\bigl(G^\vartheta_2,\alpha^\vartheta_2\bigr)\) are mutual best replies at \(\tau^\cP(\vartheta)\). For \(\{\vartheta=\tau^\cP(\vartheta)\}\) this directly implies optimality. For \(\{\vartheta<\tau^\cP(\vartheta)\}\) and any admissible \(\bigl(G^\vartheta_a,\alpha^\vartheta_a\bigr)\), time consistency and iterated expectations yield the estimate
\begin{align*}
&V^\vartheta_i\bigl(G^\vartheta_a,\alpha^\vartheta_a,G^\vartheta_j,\alpha^\vartheta_j\bigr)\\
&\leq V^\vartheta_i\bigl(G^\vartheta_a\indi{t<\tau^\cP(\vartheta)}+\underbrace{\indi{t\geq\tau^\cP(\vartheta)}}_{\mathrlap{\displaystyle=\indi{t\geq\tau^\cP(\vartheta)}\Bigl[G^\vartheta_a(\tau^\cP(\vartheta)-)+(1-G^\vartheta_a(\tau^\cP(\vartheta)-))G^\vartheta_i\Bigr]}},\alpha^\vartheta_a\indi{t<\tau^\cP(\vartheta)}+\alpha^\vartheta_i\indi{t\geq\tau^\cP(\vartheta)},G^\vartheta_j,\alpha^\vartheta_j\bigr),
\end{align*}
\(i,j\in\{1,2\}\), \(i\not=j\). Furthermore, since \(\indi{t<\tau^\cP(\vartheta)}\Delta G^\vartheta_j(t)\equiv 0\),
\begin{equation*}
V^\vartheta_i\bigl(G^\vartheta_a,\alpha^\vartheta_a,G^\vartheta_j,\alpha^\vartheta_j\bigr)=V^\vartheta_i\bigl(G^\vartheta_a,\alpha^\vartheta_a\indi{t\geq\tau^\cP(\vartheta)},G^\vartheta_j,\alpha^\vartheta_j\bigr).
\end{equation*}
Together, and using that in fact \(G^\vartheta_j(\tau^\cP(\vartheta)-)=0\), we have
\begin{align*}
&V^\vartheta_i\bigl(G^\vartheta_a,\alpha^\vartheta_a,G^\vartheta_j,\alpha^\vartheta_j\bigr)\\
&\leq V^\vartheta_{i}\bigl(G^\vartheta_a\indi{t<\tau^\cP(\vartheta)}+\indi{t\geq\tau^\cP(\vartheta)},\alpha^\vartheta_i,G^\vartheta_{j},\alpha^\vartheta_j\bigr)\\
&=E\biggl[\int_{[0,\tau^\cP(\vartheta))}L^i_s\,dG^\vartheta_{a}(s)+(1-G^\vartheta_a(\tau^\cP(\vartheta)-))V^{\tau^\cP(\vartheta)}_i\bigl(G^\vartheta_i,\alpha^\vartheta_i,G^\vartheta_j,\alpha^\vartheta_j\bigr)\biggm|\F_\vartheta\biggr]\displaybreak[0]\\
&=\int_0^1 E\biggl[L^i_{\tau^G_a(x)}\indi{\tau^G_a(x)\in[0,\tau^\cP(\vartheta))}\biggm|\F_\vartheta\biggr]\,dx\\
&+\int_0^1 E\biggl[V^{\tau^\cP(\vartheta)}_i\bigl(G^\vartheta_i,\alpha^\vartheta_i,G^\vartheta_j,\alpha^\vartheta_j\bigr)\indi{\tau^G_a(x)\in[\tau^\cP(\vartheta),\infty]}\biggm|\F_\vartheta\biggr]\,dx\\
&\leq\esssup_{\tau\geq\vartheta}E\Bigl[L^i_\tau\indi{\tau<\tau^\cP(\vartheta)}+V^{\tau^\cP(\vartheta)}_i\bigl(G^\vartheta_i,\alpha^\vartheta_i,G^\vartheta_j,\alpha^\vartheta_j\bigr)_i\indi{\tau\geq\tau^\cP(\vartheta)}\Bigm|\F_\vartheta\Bigr],
\end{align*}
where we apply a change of variable as in Lemma \ref{lem:LdG} and Fubini's Theorem in the second last step. At \(\tau^\cP(\vartheta)\), the optimal payoff is fixed at
\begin{equation*}
V^{\tau^\cP(\vartheta)}_i\bigl(G^\vartheta_i,\alpha^\vartheta_i,G^\vartheta_j,\alpha^\vartheta_j\bigr)=F^i_{\tau^\cP(\vartheta)}\indi{L^j_{\tau^\cP(\vartheta)}>F^j_{\tau^\cP(\vartheta)}}+L^i_{\tau^\cP(\vartheta)}\indi{L^j_{\tau^\cP(\vartheta)}=F^j_{\tau^\cP(\vartheta)}}.
\end{equation*} 
If \(L^i\) and \(F^i\) are continuous, \(i=1,2\), then \((L^1_{\tau^\cP(\vartheta)}-F^1_{\tau^\cP(\vartheta)})\wedge(L^2_{\tau^\cP(\vartheta)}-F^2_{\tau^\cP(\vartheta)})=0\) on \(\{\vartheta<\tau^\cP(\vartheta)\}\), a.s. Hence for the given strategies,
\begin{align*}
&V^\vartheta_i\bigl(G^\vartheta_i,\alpha^\vartheta_i,G^\vartheta_j,\alpha^\vartheta_j\bigr)\\
&=E\Bigl[V^{\tau^\cP(\vartheta)}_i\bigl(G^\vartheta_i,\alpha^\vartheta_i,G^\vartheta_j,\alpha^\vartheta_j\bigr)\Bigm|\F_\vartheta\Bigr]\\
&=E\Bigl[L^i_{\tau^\cP(\vartheta)}\indi{L^j_{\tau^\cP(\vartheta)}>F^j_{\tau^\cP(\vartheta)}}+L^i_{\tau^\cP(\vartheta)}\indi{L^j_{\tau^\cP(\vartheta)}=F^j_{\tau^\cP(\vartheta)}}\Bigm|\F_\vartheta\Bigr]\\
&=E\Bigl[L^i_{\tau^\cP(\vartheta)}\Bigm|\F_\vartheta\Bigr]=\esssup_{\tau\geq\vartheta}E\Bigl[L^i_{\tau\wedge\tau^\cP(\vartheta)}\Bigm|\F_\vartheta\Bigr]\\
&=\esssup_{\tau\geq\vartheta}E\Bigl[L^i_\tau\indi{\tau<\tau^\cP(\vartheta)}+V^{\tau^\cP(\vartheta)}_i\bigl(G^\vartheta_i,\alpha^\vartheta_i,G^\vartheta_j,\alpha^\vartheta_j\bigr)\indi{\tau\geq\tau^\cP(\vartheta)}\Bigm|\F_\vartheta\Bigr].
\end{align*}
Therefore player \(i\) has no incentive to change the strategy \(\bigl(G^\vartheta_i,\alpha^\vartheta_i\bigr)\) by placing any mass on \([\vartheta,\tau^\cP(\vartheta))\).
\end{proof}

\begin{proof}[{\bf Proof of Proposition \ref{prop:eqlpreemp}}]
The proof is analogous to that of Theorem \ref{thm:Lsubmart}, except that we use the martingale property of \(F\) in its continuation region \(\{X\in[0,x^F)\}\) instead of a submartingale property of \(L\).\footnote{
\(L\) need not be a submartingale outside the preemption region, depending on the parameter values. In particular one can show that if \(\mu\leq 0\), the drift of \(L\) is strictly negative for all \(M\) sufficiently close to 1 and \(X_t\) sufficiently close to \(x^P\).
}
Concerning the payoffs, fix \(\vartheta\in\T\). If \(X_\vartheta\geq x^P\), then the payoffs to both players are \(V^\vartheta_i\bigl(G^\vartheta_i,\alpha^\vartheta_i,G^\vartheta_j,\alpha^\vartheta_j\bigr)=F_\vartheta\) and optimality follows again from Proposition \ref{prop:eqlL>F}. If \(X_\vartheta<x^P\), i.e., \(\vartheta<\tau^\cP(\vartheta)\), then we replace the estimate in the proof of Theorem \ref{thm:Lsubmart} by
\begin{align*}
&V^\vartheta_i\bigl(G^\vartheta_a,\alpha^\vartheta_a,G^\vartheta_j,\alpha^\vartheta_j\bigr)\\
&\leq V^\vartheta_{i}\bigl(G^\vartheta_a\indi{t<\tau^\cP(\vartheta)}+\indi{t\geq\tau^\cP(\vartheta)},\alpha^\vartheta_i,G^\vartheta_{j},\alpha^\vartheta_j\bigr)\\
&=E\biggl[\int_{[0,\tau^\cP(\vartheta))}L_s\,dG^\vartheta_{i}(s)+(1-G^\vartheta_a(\tau^\cP(\vartheta)-))V^{\tau^\cP(\vartheta)}_i\bigl(G^\vartheta_i,\alpha^\vartheta_i,G^\vartheta_j,\alpha^\vartheta_j\bigr)\biggm|\F_\vartheta\biggr]\\
&=\int_0^1 E\biggl[L_{\tau^G_a(x)}\indi{\tau^G_a(x)\in[0,\tau^\cP(\vartheta))}\biggm|\F_\vartheta\biggr]\,dx\\
&+\int_0^1 E\biggl[V^{\tau^\cP(\vartheta)}_i\bigl(G^\vartheta_i,\alpha^\vartheta_i,G^\vartheta_j,\alpha^\vartheta_j\bigr)\indi{\tau^G_a(x)\in[\tau^\cP(\vartheta),\infty]}\biggm|\F_\vartheta\biggr]\,dx\displaybreak[0]\\
&\leq\int_0^1 E\biggl[F_{\tau^G_a(x)}\indi{\tau^G_a(x)\in[0,\tau^\cP(\vartheta))}\biggm|\F_\vartheta\biggr]\,dx\\
&+\int_0^1 E\biggl[F_{\tau^\cP(\vartheta)}\bigl(G^\vartheta_i,\alpha^\vartheta_i,G^\vartheta_j,\alpha^\vartheta_j\bigr)\indi{\tau^G_a(x)\in[\tau^\cP(\vartheta),\infty]}\biggm|\F_\vartheta\biggr]\,dx\\
&=E\bigl[F_{\tau^\cP(\vartheta)}\bigm|\F_\vartheta\bigr]=V^\vartheta_i\bigl(G^\vartheta_i,\alpha^\vartheta_i,G^\vartheta_j,\alpha^\vartheta_j\bigr).\qedhere
\end{align*}
\end{proof}

\subsection{Technical results}

\begin{lemma}\label{lem:classD}
A measurable process \(X=(X_t)_{t\in\R_+}\) is of class {\rm (D)} iff the set \(\{X_\tau : \tau\text{ a stopping}\) \(\text{time}\}\) is uniformly integrable for any given \(X_\infty\in L^1(P)\).
\end{lemma}

\begin{proof}
We only need to show necessity: Let \(X\) be of class {\rm (D)} and fix arbitrary \(X_\infty\in L^1(P)\) and let \(\T\) denote the set of all stopping times. Then, for any \(\tau\in\T\) and \(n\in\N\), \(\abs{X_{\tau\wedge n}}\indi{\tau<\infty}\leq\abs{X_{\tau\wedge n}}\). Hence the set \(\{\abs{X_{\tau\wedge n}}\indi{\tau<\infty} : \tau\in\T, n\in\N\}\cup\{\abs{X_\tau} : \tau<\infty\in\T\}\) is uniformly integrable as well. As we may also include limits in probability of its elements, and \(\abs{X_\tau}\indi{\tau<\infty}=\lim_{n\to\infty}\abs{X_{\tau\wedge n}}\indi{\tau<\infty}\) a.s.\ for any \(\tau\in\T\), we observe that \(\{\abs{X_\tau}\indi{\tau<\infty} : \tau\in\T\}\) is uniformly integrable.

With \(X_\infty\in L^1(P)\), also \(\{\abs{X_\infty}\indi{\tau=\infty} : \tau\in\T\}\) is uniformly integrable. Now let \(\frac{\varepsilon}{2}>0\). By uniform integrability there exists a \(\delta>0\) such that \(\max\bigl\{E[\abs{X_\tau}\indi{\tau<\infty}\indi{A}],E[\abs{X_\infty}\indi{\tau=\infty}\indi{A}]\bigr\}\leq\frac{\varepsilon}{2}\) for any measurable \(A\) with \(P(A)<\delta\) and any \(\tau\in\T\). Hence we have \(E[(\abs{X_\tau}\indi{\tau<\infty}+\abs{X_\infty}\indi{\tau=\infty})\indi{A}]\leq\varepsilon\), which shows that \(\{\abs{X_\tau}\indi{\tau<\infty}+\abs{X_\infty}\indi{\tau=\infty} : \tau\in\T\}\) is uniformly integrable as claimed.
\end{proof}

\begin{lemma}\label{lem:LdG}
If \(L\) is a (measurable) process of class {\rm (D)} then there exists a constant \(K\in\R_+\) such that for any process \(G\) that is a.s.\ right-continuous, non-decreasing, non-negative and bounded by some \(G_\infty\in L^\infty(P)\) and all random variables \(0\leq a\leq b\leq\infty\) a.s.\ we have
\begin{enumerate}
\item\label{ELdGbd}
\begin{equation*}
E\biggl[\int_{[a,b)}\abs{L_t}\,dG_t\biggr]\leq K\norm{G_\infty}_\infty<\infty
\end{equation*}
\end{enumerate}
and
\begin{enumerate}[resume]
\item\label{chgvar}
\begin{equation*}
\int_{[a,b)}\abs{L_t}\,dG_t=\int_0^\infty\bigl\lvert L_{\tau^G(x)}\bigr\rvert\indi{\tau^G(x)\in[a,b)}\,dx<\infty\quad\text{a.s.},
\end{equation*}
\end{enumerate}
where \(\tau^G(x):=\inf\{t\geq 0\mid G_t\geq x\}\), \(x\in\R_+\), and \(\Delta G_0\equiv G_0\); equivalently, ``\,\(G_t>x\)'' in \(\tau^G(x)\).

If \(\{\abs{L_\tau}\indi{\tau<\infty}:\tau\in\T\}\) is bounded in \(L^\infty(P)\) by \(K\in\R_+\) and \(G\) bounded by some \(G_\infty\in L^1(P)\), \ref{ELdGbd} holds with \(KE[G_\infty]\) instead and \ref{chgvar} as stated.
\end{lemma}

\begin{proof}
The (a.s.) non-decreasing family of stopping times \(\bigl(\tau^G(x)\bigr)_{x\in\R_+}\) is the left-continuous inverse of \(G\), which satisfies
\begin{equation*}
\tau^G(x)\leq t\ \Leftrightarrow\ G_t\geq x.
\end{equation*}
Thus, with the convention \(\int_{[0,c]}\,dG=G_c\), \(\int_{[0,\infty)}\indi{A}\,dG=\int_0^\infty\indi{\tau^G(x)\in A}\,dx\) for all \(A\in\{[0,c]:c\in\R_+\}\) and hence for \(A=\R_+\) by monotone convergence, a.s. By a monotone class argument\footnote{
See, e.g., \cite{Kallenberg02}, Theorem 1.1. 
}
the relation holds on all of \(\B(\R_+)\) a.s.

Since \(L_\cdot(\omega):\R_+\to\R\), \(t\mapsto L_t(\omega)\), is Borel measurable\footnote{
See, e.g., \cite{Kallenberg02}, Lemma 1.26 (i). 
} 
like the function \(\indi{t\in[a(\omega),b(\omega))}\), we now obtain the following change-of-variable formula\footnote{
See, e.g., \cite{Kallenberg02}, Lemma 1.22. One needs to restrict \(dx\) to \(\{\tau^G(x)<\infty\}\), which is redundant whenever we have \([a,b)\).
}:
\begin{equation*}
\int_{[a,b)}\abs{L_t}\,dG_t=\int_{\{\tau^G(x)<\infty\}}\bigl\lvert L_{\tau^G(x)}\bigr\rvert\indi{\tau^G(x)\in[a,b)}\,dx \qquad\text{a.s.} 
\end{equation*}
As \(\inf\{t\geq 0\mid G_t>x\}=\tau^G(x+)\), which differs from \(\tau^G(x)\) only on a set of Lebesgue measure (\(dx\)) 0, we can equivalently use the former. By Fubini's Theorem
\begin{equation}\label{ELdG}
E\biggl[\int_0^\infty\bigl\lvert L_{\tau^G(x)}\bigr\rvert\indi{\tau^G(x)\in[a,b)}\,dx\biggr]\leq\int_0^{\norm{G_\infty}_\infty}E\Bigl[\bigl\lvert L_{\tau^G(x)}\bigr\rvert\indi{\tau^G(x)<\infty}\Bigr]\,dx.
\end{equation}
As \(L\) is of class {\rm (D)}, \(\{\abs{L_\tau}\indi{\tau<\infty}:\tau\in\T\}\) is bounded in \(L^1(P)\) by some \(K<\infty\), whence the RHS of \eqref{ELdG} is bounded by \(K\norm{G_\infty}_\infty\) if the latter is finite. If \(\sup_{\tau\in\T}\norm{L_\tau\indi{\tau<\infty}}_\infty\leq K\) and \(G\) bounded by \(G_\infty\in L^1(P)\), then the RHS of \eqref{ELdG} is bounded by
\begin{equation*}
\int_0^\infty E\Bigl[K\indi{\tau^G(x)<\infty}\Bigr]\,dx=E\biggl[\int_0^\infty K\indi{\tau^G(x)<\infty}\,dx\biggr]\leq KE[G_\infty]<\infty.
\end{equation*}
In either case it follows that \(\int_{[a,b)}\abs{L_t}\,dG_t<\infty\) a.s.
\end{proof}

\begin{lemma}\label{lem:oalpharc}
Suppose \(\alpha\) is a progressively measurable process and \({}^o\alpha\) its optional projection. Let \(\tau\in\T\) be given. Then \({}^o\alpha\) is a.s.\ right-continuous at \(\tau<\infty\) where \(\alpha\) is so.
\end{lemma}

\begin{proof}
For any \(\varepsilon>0\) define \(\tau_\varepsilon:=\inf\{t\geq\tau:\abs{\alpha_t-\alpha_\tau}>\varepsilon\}\in\T\) as \(\alpha\) is progressive. Then the set \(B_\varepsilon:=\{(\omega,t)\mid\tau\leq t<\infty,{}^o\alpha_t-{}^o\alpha_\tau>\varepsilon\}\cap[\tau,\tau_\varepsilon)\) is optional and \(P[\sigma\in B_\varepsilon]=0\) for any \(\sigma\in\T\) because \(\alpha_\sigma={}^o\alpha_\sigma\) on \(\{\sigma<\infty\}\) a.s. Hence, if we denote by \(A_\varepsilon\) the canonical projection of \(B_\varepsilon\) onto \(\Omega\), \(P[A_\varepsilon]=0\) by the optional section theorem (\cite{DellacherieMeyer78}, Theorem IV.84), and therefore \(A^c:=\bigcap_{n\in\N}A^c_{1/n}\) is an a.s.\ event with \((A^c\times\R_+)\cap B_{1/n}=\emptyset\), \(n\in\N\). By switching signs we obtain the same result for \(\abs{{}^o\alpha_t-{}^o\alpha_\tau}>1/n\) in \(B_{1/n}\) (we do not rename any sets). 

Now, given any \(\omega\in A^c\) for which \(\alpha\) is right-continuous at \(\tau\), we must have \(\tau_{1/\ceil{\varepsilon^{-1}}}(\omega)>\tau\) for any \(\varepsilon>0\), such that \(\abs{{}^o\alpha_t-{}^o\alpha_\tau}\leq1/\ceil{\varepsilon^{-1}}\leq\varepsilon\) on \([\tau(\omega),\tau_{1/\ceil{\varepsilon^{-1}}}(\omega))\not=\emptyset\).
\end{proof}

\subsection{Supplementary results}

\begin{lemma}\label{lem:alphapayoff}
Suppose \(\vartheta=\hat\tau^\vartheta_1=\hat\tau^\vartheta_2\) and \(0=\alpha^\vartheta_i(\vartheta)<\alpha^\vartheta_j(\vartheta)<1\). Then \(\frac{\lambda^\vartheta_M}{\lambda^\vartheta_{L,i}}=\frac{\alpha^\vartheta_j(\vartheta)}{1-\alpha^\vartheta_j(\vartheta)}\).
\end{lemma}

\begin{proof}
We introduce the function \(\mu_F(x,y):=\mu_L(y,x)=1-\mu_L(x,y)-\mu_M(x,y)\) and use the short-hand notation 
\begin{flalign*}
&& \underline{\mu_\cdot}&=\liminf_{\underset{\alpha^\vartheta_i(t)+\alpha^\vartheta_j(t)>0}{t\searrow\vartheta}}\mu_\cdot(\alpha^\vartheta_i(t),\alpha^\vartheta_j(t)),\quad\overline{\mu_\cdot}=\limsup_{\underset{\alpha^\vartheta_i(t)+\alpha^\vartheta_j(t)>0}{t\searrow\vartheta}}\mu_\cdot(\alpha^\vartheta_i(t),\alpha^\vartheta_j(t)), &&\\
&& \underline{\alpha_i}&=\liminf_{t\searrow\vartheta}\alpha^\vartheta_i(t),\quad\overline{\alpha_i}=\limsup_{t\searrow\vartheta}\alpha^\vartheta_i(t)\quad\text{and}\quad\alpha_j=\alpha^\vartheta_j(\vartheta).&&
\end{flalign*}
In our current case \(\vartheta=\hat\tau^\vartheta_1=\hat\tau^\vartheta_2\) and \(0=\alpha^\vartheta_i(\vartheta)<\alpha^\vartheta_j(\vartheta)<1\), we now have
\begin{flalign*}
&& \lambda^\vartheta_{L,i}&=(1-\alpha_j)\frac12(\underline{\mu_L}+\overline{\mu_L}), &&\\
&& \lambda^\vartheta_{L,j}&=\alpha_j+(1-\alpha_j)\frac12(\underline{\mu_F}+\overline{\mu_F}) &&\\
\text{and thus}\\
&& \lambda^\vartheta_M&=1-\lambda^\vartheta_{L,i}-\lambda^\vartheta_{L,j} &&\\
&& &=(1-\alpha_j)\frac12(2-\underline{\mu_L}-\overline{\mu_L}-\underline{\mu_F}-\overline{\mu_F}). &&
\end{flalign*}
Using \(\alpha_j=\lim_{t\searrow\vartheta}\alpha^\vartheta_j(t)\) and the continuity and monotonicity of \(\mu_L\) and \(\mu_F\) we obtain
\begin{align*}
\frac{\lambda^\vartheta_M}{\lambda^\vartheta_{L,i}}&=\frac{2-\underline{\mu_L}-\overline{\mu_L}-\underline{\mu_F}-\overline{\mu_F}}{\underline{\mu_L}+\overline{\mu_L}}\\
&=\frac{2-\mu_L(\underline{\alpha_i},\alpha_j)-\mu_L(\overline{\alpha_i},\alpha_j)-\mu_F(\overline{\alpha_i},\alpha_j)-\mu_F(\underline{\alpha_i},\alpha_j)}{\mu_L(\underline{\alpha_i},\alpha_j)+\mu_L(\overline{\alpha_i},\alpha_j)}\\
&=\frac{\mu_M(\underline{\alpha_i},\alpha_j)+\mu_M(\overline{\alpha_i},\alpha_j)}{\mu_L(\underline{\alpha_i},\alpha_j)+\mu_L(\overline{\alpha_i},\alpha_j)}=\frac{\alpha_j}{1-\alpha_j}. \qedhere
\end{align*}
\end{proof}

\begin{lemma}\label{lem:limit}
Fix \(\sigma\in\T\) and suppose \(\bigl(G_1,\alpha_1\bigr)\) and \(\bigl(G_2,\alpha_2\bigr)\) are time-consistent extended mixed strategies which induce an equilibrium in all subgames \(\Gamma^\vartheta\) beginning at some \(\vartheta\in\T\) taking values in \((\sigma,\infty]\) a.s. 

If \(\sigma=\hat\tau^\sigma_1=\hat\tau^\sigma_2\) and \(\lim_{t\searrow\sigma}\alpha^\sigma_1(t)=\lim_{t\searrow\sigma}\alpha^\sigma_2(t)=0\), then 
\begin{equation*}
L^1_\sigma-F^1_\sigma=L^2_\sigma-F^2_\sigma=0\qquad\text{a.s.}
\end{equation*}
\end{lemma}

\noindent
Note that we do not impose right-continuity of any \(\alpha^\sigma_i\) at \(\sigma\) in the lemma.

\begin{proof}
Let \(i,j\in\{1,2\}\), \(i\not=j\). Suppose first \(L^i_\sigma>F^i_\sigma\). By hypothesis there exist arbitrarily small right (random) neighbourhoods of \(\sigma\) in which \(\alpha^\sigma_i\) is bounded away from 1, in which \(\alpha^\sigma_j\) takes some strictly positive values, and in which \(L^i>F^i\)\,(\(\geq M^i\)). In any such neighbourhood we must have by condition \eqref{indiff} that
\begin{equation*}
\alpha^\sigma_j>0\Rightarrow\alpha^\sigma_j\geq\frac{L^i-F^i}{L^i-M^i},
\end{equation*}
implying \(\limsup_{t\searrow\sigma}\frac{L^i-F^i}{L^i-M^i}=0\) and therefore \(L^i_\sigma=F^i_\sigma\).

Now suppose \(L^i_\sigma<F^i_\sigma\). Then in any right (random) neighbourhood of \(\sigma\) in which \(L^i<F^i\) and \(\alpha^\sigma_j\) is bounded away from 1, \(\alpha^\sigma_i\) can only be strictly positive where \(\alpha^\sigma_j=0\), i.e., the supports of \(\alpha^\sigma_i\) and \(\alpha^\sigma_j\) in these neighbourhoods must be disjoint. Hence, whenever \(\alpha^\sigma_i>0\), player \(i\) becomes leader and must prefer so over becoming follower at the next time when \(\alpha^\sigma_j>0\). Now consider neighbourhoods between \(\sigma\) and
\begin{equation*}
\sigma':=\inf\bigl\{t\geq\sigma\mid L^i_t-L^i_\sigma\geq(F^i_\sigma-L^i_\sigma)/3\text{ or }F^i_t-F^i_\sigma\leq-(F^i_\sigma-L^i_\sigma)/3\bigr\}>\sigma.
\end{equation*}
At any stopping time \(\vartheta\in[\sigma,\sigma']\), \(i\) can only prefer to stop if \(\alpha^\sigma_j=0\) on \([\sigma,\sigma']\), which contradicts the hypothesis.
\end{proof}

\newpage

\begin{thebibliography}{}

\bibitem[\protect\citeauthoryear{Al{\'o}s-Ferrer and
  Ritzberger}{Al{\'o}s-Ferrer and Ritzberger}{2008}]{Alos-FerrerRitzberger08}
Al{\'o}s-Ferrer, C. and K.~Ritzberger (2008).
\newblock Trees and extensive forms.
\newblock {\em J.\ Econ.\ Theory\/}~{\em 143}, 216--250.

\bibitem[\protect\citeauthoryear{Dellacherie and Meyer}{Dellacherie and
  Meyer}{1978}]{DellacherieMeyer78}
Dellacherie, C. and P.-A. Meyer (1978).
\newblock {\em Probabilities and Potential}, Volume~29 of {\em Mathematics
  Studies}.
\newblock Amsterdam New York Oxford: North-Holland.

\bibitem[\protect\citeauthoryear{El~Karoui}{El~Karoui}{1981}]{ElKaroui81}
El~Karoui, N. (1981).
\newblock Les aspects probabilistes du contr{\^o}le stochastique.
\newblock In P.-L. Hennequin (Ed.), {\em Ecole d'Et{\'e} de Probabilit{\'e}s de
  Saint-Flour IX-1979}, Volume 876 of {\em Lecture Notes in Math.}, pp.\
  73--238. Berlin Heidelberg New York: Springer.

\bibitem[\protect\citeauthoryear{Fudenberg and Tirole}{Fudenberg and
  Tirole}{1985}]{FudenbergTirole85}
Fudenberg, D. and J.~Tirole (1985).
\newblock Preemption and rent equalization in the adoption of new technology.
\newblock {\em Rev.\ Econ.\ Stud.\/}~{\em 52}, 383--401.

\bibitem[\protect\citeauthoryear{Hamad{\`e}ne and Hassani}{Hamad{\`e}ne and
  Hassani}{2014}]{HamadeneHassani14}
Hamad{\`e}ne, S. and M.~Hassani (2014).
\newblock The multi-player nonzero-sum {D}ynkin game in continuous time.
\newblock {\em SIAM J.\ Control Optim.\/}~{\em 52\/}(2), 821--835.

\bibitem[\protect\citeauthoryear{Hamad{\`e}ne and Zhang}{Hamad{\`e}ne and
  Zhang}{2010}]{HamadeneZhang10}
Hamad{\`e}ne, S. and J.~Zhang (2010).
\newblock The continuous-time nonzero-sum {D}ynkin game problem and application
  in game options.
\newblock {\em SIAM J.\ Control Optim.\/}~{\em 48\/}(5), 3659--3669.

\bibitem[\protect\citeauthoryear{Hendricks, Weiss, and Wilson}{Hendricks
  et~al.}{1988}]{Hendricksetal88}
Hendricks, K., A.~Weiss, and C.~Wilson (1988).
\newblock The war of attrition in continuous time with complete information.
\newblock {\em Int.\ Econ.\ Rev.\/}~{\em 29\/}(4), 663--680.

\bibitem[\protect\citeauthoryear{Hendricks and Wilson}{Hendricks and
  Wilson}{1992}]{HendricksWilson92}
Hendricks, K. and C.~Wilson (1992).
\newblock Equilibrium in preemption games with complete information.
\newblock In M.~Majumdar (Ed.), {\em Equilibrium and Dynamics: Essays in Honour
  of {D}avid {G}ale}, pp.\  123--147. Basingstoke, Hampshire: Macmillan.

\bibitem[\protect\citeauthoryear{Hoppe and Lehmann-Grube}{Hoppe and
  Lehmann-Grube}{2005}]{HoppeLehmann-Grube05}
Hoppe, H.~C. and U.~Lehmann-Grube (2005).
\newblock Innovation timing games: a general framework with applications.
\newblock {\em J.\ Econ.\ Theory\/}~{\em 121}, 30--50.

\bibitem[\protect\citeauthoryear{Kallenberg}{Kallenberg}{2002}]{Kallenberg02}
Kallenberg, O. (2002).
\newblock {\em Foundations of Modern Probability\/} (2nd ed.).
\newblock New York, Berlin, Heidelberg: Springer.

\bibitem[\protect\citeauthoryear{Laraki and Solan}{Laraki and
  Solan}{2013}]{LarakiSolan13}
Laraki, R. and E.~Solan (2013).
\newblock Equilibrium in two-player non-zero-sum {D}ynkin games in continuous
  time.
\newblock {\em Stochastics\/}~{\em 85\/}(6), 997--1014.

\bibitem[\protect\citeauthoryear{Laraki, Solan, and Vieille}{Laraki
  et~al.}{2005}]{Larakietal05}
Laraki, R., E.~Solan, and N.~Vieille (2005).
\newblock Continuous-time games of timing.
\newblock {\em J.\ Econ.\ Theory\/}~{\em 120}, 206--238.

\bibitem[\protect\citeauthoryear{Mason and Weeds}{Mason and
  Weeds}{2010}]{MasonWeeds10}
Mason, R. and H.~Weeds (2010).
\newblock Investment, uncertainty and pre-emption.
\newblock {\em Int.\ J.\ Ind.\ Organ.\/}~{\em 28\/}(3), 278--287.

\bibitem[\protect\citeauthoryear{Pawlina and Kort}{Pawlina and
  Kort}{2006}]{PawlinaKort06}
Pawlina, G. and P.~M. Kort (2006).
\newblock Real options in an asymmetric duopoly: Who benefits from your
  competitive disadvantage?
\newblock {\em J.\ Econ.\ Manage.\ Strategy\/}~{\em 15\/}(1), 1--35.

\bibitem[\protect\citeauthoryear{Revuz and Yor}{Revuz and Yor}{1999}]{RevuzYor}
Revuz, D. and M.~Yor (1999).
\newblock {\em Continuous Martingales and {B}rownian Motion\/} (3rd ed.).
\newblock Berlin Heidelberg New York: Springer.

\bibitem[\protect\citeauthoryear{Steg and Thijssen}{Steg and
  Thijssen}{2014}]{StegThijssen14}
Steg, J.-H. and J.~J.~J. Thijssen (2014).
\newblock A game of switching options.
\newblock mimeo.

\bibitem[\protect\citeauthoryear{Thijssen, Huisman, and Kort}{Thijssen
  et~al.}{2012}]{Thijssenetal12}
Thijssen, J. J.~J., K.~J.~M. Huisman, and P.~M. Kort (2012).
\newblock Symmetric equilibrium strategies in game theoretic real option
  models.
\newblock {\em J.\ Math.\ Econ.\/}~{\em 48\/}(4), 219--225.

\bibitem[\protect\citeauthoryear{Touzi and Vieille}{Touzi and
  Vieille}{2002}]{TouziVieille02}
Touzi, N. and N.~Vieille (2002).
\newblock Continuous-time {D}ynkin games with mixed strategies.
\newblock {\em SIAM J. Control Optim.\/}~{\em 41}, 1073--1088.

\bibitem[\protect\citeauthoryear{Weeds}{Weeds}{2002}]{Weeds02}
Weeds, H. (2002).
\newblock Strategic delay in a real options model of {R}\&{D} competition.
\newblock {\em Rev.\ Econ.\ Stud.\/}~{\em 69}, 729--747.

\end{thebibliography}

\end{document}